\documentclass[a4paper,12pt]{article}
\usepackage{amsfonts,amssymb}
\usepackage{latexsym}
\usepackage{mathrsfs}
\usepackage[colorlinks=true,linkcolor=red]{hyperref}

\oddsidemargin
\topmargin
  \usepackage{amsthm}
  \usepackage{amsmath}

\usepackage{amssymb}
  
  \newtheorem{theorem}{Theorem}[section]
  \newtheorem{corollary}[theorem]{Corollary}
  
  \newtheorem{lemma}[theorem]{Lemma}

  \theoremstyle{definition}
  \newtheorem{definition}[theorem]{Definition}
  \newtheorem{remark}[theorem]{Remark}
  \newtheorem{example}[theorem]{Example}

  \numberwithin{equation}{section}

\newcommand\restr[2]{{% we make the whole thing an ordinary symbol
  \left.\kern-\nulldelimiterspace % automatically resize the bar with \right
  #1 % the function
  \littletaller % pretend it's a little taller at normal size
  \right|_{#2} % this is the delimiter
  }}

\newcommand{\littletaller}{\mathchoice{\vphantom{\big|}}{}{}{}}

  \title{\sc Carath\'eodory-type selection and random fixed point theorems for discontinuous correspondences}
 \author{
{\bf Anuj Bhowmik}\thanks{Indian Statistical Institute, 203 B.T. Road, Kolkata 711108, India. 
 }\\
Indian Statistical Institute\\
\vspace{0.4cm}
e-mail: anujbhowmik09@gmail.com \\
{\bf Nicholas C. Yannelis}\thanks{Department of Economics, The University of Iowa, Iowa City, IA, USA 
 }\\
The University of Iowa\\
\vspace{0.4cm}
e-mail: nicholasyannelis@gmail.com } 

\date{\today}

\begin{document}

  \maketitle

\begin{abstract}
Research in Economics and Game theory has necessitated results on Carath\'eodory-type selections. In particular, one has to obtain Carath\'eodory type-selections from correspondences that need not be continuous (neither lower-semicontinuous nor upper-semicontinuous). We provide new theorems on Carath\'eodory type-selections that include as corollaries the results in Kim-Prikry-Yannelis \cite{KPY:87}. We also, obtain new random fixed-point theorems, random maximal elements, random (Nash) equilibrium and Bayesian equilibrium extending and generalizing theorems of Browder 
\cite{Browder:68}, Fan \cite{Fan:52} and Nash \cite{Nash}, among others.

\end{abstract}

{\bf Keywords:} 

\section{Introduction}

Numerous prominent games in economic theory--such as the Hotelling location game, Bertrand competition, Cournot competition with fixed costs, and various auction frameworks--are characterized by discontinuous payoff functions. Consequently, they fail to satisfy the conditions stipulated by Nash’s original existence theorem and its extensions to infinite-dimensional strategy spaces.
Despite this, such games frequently admit the existence of at least one pure strategy Nash equilibrium. The seminal 
work of Dasgupta-Maskin \cite{DM} and Reny \cite{Reny:99} allow discontinuity in 
the payoff functions and subsumes earlier equilibrium existence result by Nash \cite{Nash}
covering many such examples. We refer to a survey article by Reny  \cite{Reny:20} and 
some recent papers \cite{Khan1, Khan2, Khan3} for the recent development of this research area. 
Note that all of these works are limited to 
finitely many players, whereas games involving a continuum of players constitute a separate 
branch of research in the literature (see for example, \cite{HeSun, Khan, KhanSun}).
This framework is particularly useful in economics and social sciences, where the population is large 
enough that individual actions have negligible impact on the overall outcome.

\medskip
Recent advances have been made in establishing the existence of competitive equilibria in economies with finitely many agents and discontinuous preferences, as demonstrated in the works of He-Yannelis \cite{HY:2017}, Khan et al. \cite{Khan1,Khan2}, Podczeck-Yannelis \cite{Podyan}, and Reny \cite{Reny:16}, among others. These contributions have shown that the strong assumptions traditionally required for proving equilibrium existence in earlier foundational works-such as those by Arrow and Debreu \cite{AD}, McKenzie \cite{Mc:54, Mc:59, Mc:81}, Mas-Colell \cite{Mas-colell},
Shafer and Sonnenschein \cite{SS}, and Shafer \cite{Shafer}--can, in fact, be relaxed.
However, when extending the analysis to economies with a continuum of agents and discontinuous preferences—a framework that more accurately captures the essence of perfect competition, where each individual agent has a negligible effect on market prices—new mathematical techniques become necessary. In particular, the recent work by Bhowmik-Yannelis \cite{BY} highlights the need for Carathéodory-type selection theorems for correspondences that are neither upper nor lower semicontinuous.
In the following section, we outline how this mathematical problem naturally arises within the context of economic theory.

\medskip
Consider a correspondence $F$ defined on the product space $T\times X$ ($T$ is a measurable space and $X$ is a topological space) taking values on a Banach space $Y$. A 
Carath\'eodory-type selection is a function $f:T\times X\to Y$ such that $f(t,x)$ belongs to $F(t,x)$ for all $(t,x) \in T\times X$ and it is  continuous in $x$ and
measurable in $t$. The correspondence $F$ is typically assumed to be lower-semicontinuous (l.s.c.) in $x$ and also jointly measurable. As it was pointed out in 
Kim-Prikry-Yannelis (KPY) \cite{KPY:87},\footnote{See also Kim-Prikry-Yannelis \cite{KPY:88} for a similar result.} Carath\'eodory-type selections generalize the Michael \cite{Michael} continuous selections theorems. Such selections arise naturally in economic problems. 
However, in economic applications the domain of the correspondence $F$ is not defined on a ``product space”, $T\times X$, but on an arbitrary subset of $T\times X$. 
The later complicates the proofs considerably. KPY \cite{KPY:87} proved a Carath\'eodory-type selection theorem for a correspondence whose 
domain was an arbitrary subset of a product space.

\medskip
The main purpose of this paper is to go beyond the KPY results and obtain Carath\'eodory-type selections from correspondences that may not be l.s.c., and thus generalize the KPY theorems as well as other related results in the literature (this is shown in Section \ref{sec:main}). To achieve this, we generalized the continuous inclusion property of He-Yannelis \cite{HY:2017} to a two-variable framework (one argument belongs to a measure space and the other belongs to a metric space), where the correspondence may not be l.s.c. However, the relaxation of the continuity assumption on the correspondence necessitates new arguments  and details that are far from trivial. This selection theorem also enables us to establish a random fixed point theorem. 
Furthermore, we provide applications of our new Carath\'eodory-type selection (and fixed-point) theorems. In particular, we prove new results on the existence of an equilibrium of a large abstract economy, random maximal elements, random (Nash) equilibrium and Bayesian equilibrium and thus extending and generalizing theorems of  Bhowmik-Yannelis \cite{BY}, Browder \cite{Browder:68}, Fan \cite{Fan:52}, Khan \cite{Khan}, Kim-Yannelis \cite{KY}, Nash \cite{Nash}, and Yannelis \cite{Yannelis:91a}, among others.

\section{Preliminaries}\label{sec:pre}

{\bf 2.1 Notations and conventions:} 

\begin{itemize}
\item $\emptyset$ \hspace{9pt} denotes the empty set, 
\item $2^A$ \hspace{9pt} denotes the power set of $A$,
\item $\mathbb R$ \hspace{12pt} denotes the set of real numbers,
\item ${\rm cl}A$ \hspace{5pt} denotes the closure of the set $A$,
\item ${\rm con}A$ denotes the convex hull of the set $A$,
\item $\overline{\rm con}A$ denotes the closed convex hull of the set $A$,
\item ${\rm bd}A$ \hspace{5pt} denotes the boundary of the set $A$,
\item $\setminus$ \hspace{16pt} denotes the set-theoretic subtraction,
\item ${\rm proj}$ \hspace{3pt} denotes projection,
\item If $\Phi: X \to 2^Y$ is a correspondence then $\Phi_{|U}: U \to 2^Y$ denotes the restriction of 
$\Phi$ to $U$.
\end{itemize}

\medskip
\noindent
{\bf 2.2 Definitions:} Let $X, Y$ be two topological spaces. A correspondence $\Phi: X \to 2^Y$ is said to be {\it upper-semicontinuous} 
({\it u.s.c.}) if the set $\left\{x \in X: \Phi(x) \subseteq V\right\}$ is open in $X$ for every open subset $V$ of $Y$. The 
{\it graph} of the correspondence $\Phi: X \to 2^Y$ is defined as 
\[
G_{\Phi}= \left\{(x,y) \in X \times Y: y \in \Phi(x)\right\}.
\] 
A correspondence $\Phi: X \to 2^Y$ is said to have a \emph{closed graph} if the set $G_{\Phi}$ is closed in $X \times Y$. 
A correspondence $\Phi: X \to 2^Y$ is said to be \emph{lower-semicontinuous} ({\it l.s.c.}) if the set $\left\{x \in X: \Phi(x) \cap V \neq \emptyset\right\}$ is open in $X$ for every open subset $V$ of $Y$. A correspondence $\Phi: X \to 2^Y$ is said to have \emph{open lower sections} if for each $y \in Y$, the set $\Phi^{-1}(y) = \left\{x \in X: y \in \Phi(x)\right\}$ is open in $X$. If 
 $\Phi(x)$ is open in $Y$ for each $x \in X$, $\Phi$ is said to have \emph{open upper sections}. Define 
$\mathscr P_0(X)=2^X\setminus \{\emptyset\}$. For any nonempty open set $W$ in $X$, we define\footnote{We assume that $\emptyset^u=\emptyset^\ell=\emptyset$.}
\[
W^{\rm u}=\{A\in \mathscr P_0 (X): A\subseteq W\} \mbox{ and } W^{\ell}=\{A\in \mathscr P_0 (X): A\cap W\neq \emptyset\}.
\] 
Notice that $\{W^{\rm u}: W \mbox{ open in } X\}$ forms a base for a topology on $\mathscr P_0 (X)$. This topology is called 
an \emph{upper topology} on $\mathscr P_0 (X)$ and is denoted by $\tau_u$. On the other hand, $\{W^{\ell}: W \mbox{ open in } X\}$ 
forms a subbase for a topology on $\mathscr P_0 (X)$. This topology is termed as a \emph{lower topology} on $\mathscr P_0 (X)$ and is 
denoted by $\tau_\ell$. The collection of sets of the form
\[
W_0^u\cap W_1^\ell\cap\cdots\cap W_n^\ell,
\]
where $W_0,\cdots,W_n$ are open subsets of $X$, is closed under finite intesections. Since $X^u=X^\ell=\mathscr P_0(X)$, it thus forms 
a base for a topology on $2^X$. This topology is known as the 
\emph{Vietoris topology} and is denoted by $\tau_V$. We 
are most interested in the relativization of this topology to the space of nonempty closed subsets of $X$, denoted by $\mathscr F_0(X)$
and the space of nonempty compact susbets of $X$, denoted by $\mathscr K_0(X)$. 
For any nonempty valued correpondence $\Phi: X \to 2^Y$, the following holds: 
\begin{itemize}
\item[(i)] $\Phi$ is upper-semicontinuous if and only if it is a continuous function from $X$ to 
$(\mathscr P_0(Y), \tau_u)$. 
\item[(ii)] $\Phi$ is lower-semicontinuous if and only if it is a continuous function from $X$ to 
$(\mathscr P_0(Y), \tau_\ell)$.
\item[(iii)] $\Phi$ is continuous if and only if it is a continuous function from $X$ to 
$(\mathscr P_0(Y), \tau_V)$.
\end{itemize}

\medskip
Let $(X,d)$ be a metric space. For any $A, B\in \mathscr F_0(X)$, define
 \[
 H(A, B)= \max\left\{\sup_{a\in A} {\rm dist}(a, B), \
 \sup_{b\in B}{\rm dist}(b, A)\right\},
 \]
 where
 \[
 {\rm dist}(a, B)= \inf_{b\in B}d(a, b).
 \]
 It can be readily checked that $H:\mathscr F_0(X)\times \mathscr F_0(X)\to \mathbb R_+$
 is a metric on $\mathscr F_0(X)$, called the
 \emph{Hausdorff metric}. For any $A \subseteq X$
 and $\varepsilon> 0$, let
 \[
 N_\varepsilon (A)= \left\{x\in X: {\rm dist}(x,
 A)< \varepsilon\right\}.
 \]
 The metric $H$ can also be expressed as
 \[
 H(A, B)= \sup\left\{\left|{\rm dist}(x, A)- {\rm dist}(x,
 B)\right|:x\in X\right\},
 \]
 or
 \[
 H(A, B)= \inf\left\{\varepsilon> 0: A\subseteq
 N_\varepsilon(B) \mbox{ and } B\subseteq N_\varepsilon(A)
 \right\}
 \]
 for $A, B\in \mathscr F_0(X)$. The topology
 $\tau_H$ on ${\mathscr F}_0 (X)$,
 induced by $H$, is called the \emph{Hausdorff metric
 topology}. If the function $\Phi: X \to \left({\mathscr F}_0 (Y), {\tau}_H\right)$ is continuous, then 
the correspondence $\Phi:X \to 2^Y$ is called \emph{Hausdorff
 continuous}. Let $\{A_n:n\ge 1\}\subseteq \mathscr
 P_0(X)$. A
 point $x\in X$ is called a \emph{limit point}
 of $\{A_n:n\ge 1\}$ if there exist $N\ge 1$ and points
 $x_n\in A_n$ for each $n \ge N$ such that $\{x_n:n \ge N\}$
 converges to $x$. The set of limit points of $\{A_n:n
 \ge 1\}$ is denoted by ${\rm Li}A_n$. Similarly, a point
 $x\in X$ is called a \emph{cluster point} of
 $\{A_n:n\ge 1\}$ if there exist positive integers $n_1< n_2
 <\cdots$ and for each $k$ a point $x_k \in A_{n_k}$ such
 that $\{x_k:k\ge 1\}$ converges to $x$. The set of
 cluster points of $\{A_n:n\ge 1\}$ is denoted by ${\rm Ls}A_n$. It is clear that 
 ${\rm Ls}A_n\subseteq {\rm Ls}
 A_n$, and both ${\rm Li}A_n$ and ${\rm Ls}A_n$ are closed
 (possibly empty) sets. If ${\rm Ls}A_n \subseteq {\rm Li}
 A_n$ then ${\rm Li}A_n= {\rm Ls}A_n= A$ is called the
 \emph{limit} of the sequence $\{A_n:n\ge 1\}$. Note that
 ${\rm Ls}A_n= {\rm Ls}\ {\rm cl}A_n$ and ${\rm Li}A_n=
 {\rm Li}\ {\rm cl}A_n$.

\medskip
Let $(T, \mathscr T, \mu)$ be a complete finite measure space, i.e., $\mu$ is a real-valued, non-negative, countably additive measure defined on a complete $\sigma$-field 
$\mathscr T$ of subsets of $T$ such that $\mu(T) < \infty$. We say that  $(T, \mathscr T, \mu)$ is \emph{atomic} if it can be written as the union of countably many atoms. 
Throughout the rest of this section, we assume that $Y$ is a Banach space. We denote by $L_{1}(\mu, Y)$ the space of 
equivalence classes of $Y$-valued Bochner integrable functions $f: T \to Y$ normed by 
$\|f\| = \int_{T}\|f(t)\|d\mu(t)$.
A correspondence $\Phi: T \to 2^Y$ is said to be \emph{integrably bounded} if there exists an element
$g \in L_{1}(\mu, \mathbb R)$ such that 
\[
{\rm sup}\left\{\|x\|: x \in \Phi(t)\right\} \leq g(t) \, \mu\mbox{-a.e.} 
\]
A correspondence $\Phi: T \to 2^Y$ is said to have a \emph{measurable graph} if 
$G_{\Phi} \in {\mathscr T} \otimes \mathscr{B}(Y)$, where $\mathscr{B}(Y)$ denotes 
the Borel $\sigma$-algebra on $Y$ and $\otimes$ denotes $\sigma$-product field.
A correspondence $\Phi: T \to 2^Y$ is said to be \emph{lower measurable} if 
$\left\{t \in T: \Phi(t) \cap V \neq \emptyset\right\} \in \mathscr T$ for every open subset $V$ of $Y$. It is worth pointing out 
that if $\mathscr T$ is complete and a correspondence $\Phi: T \to 2^Y$ has a measurable graph, then 
$\Phi$ is lower measurable. Moreover, if if a correspondence $\Phi: T \to 2^Y$ is closed valued and lower measurable, 
then $\Phi$ has a measurable graph. For an extensive survey, see Yannelis \cite{Yannelis:91a}. 

\medskip
Let now, $X$ be a topological space. Let $\Phi: X \to 2^Y$ be a nonempty valued correspondence. 
A function $\varphi: X \to Y$ is said to be a \emph{continuous selection} from $\Phi$ if $\varphi(x) \in \Phi(x)$ for all $x \in X$, and 
$\varphi$ is continuous. Let $(T, \mathscr T, \mu)$ be an arbitrary measure space. Let $\Psi: T 
\to 2^Y$ be a non-empty valued correspondence. A function $\psi: T \to \mathbb R^\ell$ is said to be a \emph{measurable selection} 
of $\Psi$ if $\psi(t) \in \Psi(t)$ for all $t \in T$, and $\psi$ is measurable. We denote by $\mathcal S^1_\Psi$ the set of integrable 
selections of $\Psi$, i.e.,
\[
\mathcal S^1_\Psi:=\left\{\psi\in L_1(\mu, Y): \psi \mbox{ is a measurable selection of } \Psi\right\}.
\]
For any correpondence $\Psi:T\to 2^Y$, the \emph{integral} of $\Psi$ is defined by 
\[
\int_T \Psi \, d\mu=\left\{\int_T \psi \, d\mu: \psi\in \mathcal S_\Psi^1\right\}.
\] 
{\bf Aumann Measurable Selection Theorem:} Let $(T, \mathscr T, \mu)$ be a complete finite measure space, 
$Y$ be a complete, separable metric space and $\Psi: T\to 2^Y$ be a nonempty valued correspondence with 
a measurable graph. Then there is a measurable function $\psi:T\to Y$ such that $\psi(t)\in \Psi(t)$ $\mu$-a.e.

\medskip
\noindent
{\bf Kuratowski and Ryll-Nardzewski Measurable Selection Theorem:} Let $(T,\mathscr T)$ be a measurable space, $Y$ 
be a separable metric space and $\Psi:T\to 2^Y$ be a lower measurable, closed, nonempty valued
correspondence. Then there exists a measurable selection $\psi$ from $\Psi$. 

\medskip
Note that, in the Aumann measurable slection theorem, the completeness of  $(T, \mathscr T, \mu)$ implies that $\psi(t)\in \Psi(t)$ for all 
$t\in T$, which means that $\psi$ is infact a measurable selection of $\Psi$. Futhermore, none of the above measurable selection theorem 
implies the other and under the hypothesis of any one the above measurable selection theorems, the integral  of $\Psi$ is nonempty.

\medskip
In what follows, we define the concept of a Carath\'eodory-type selection which roughly speaking combines the notions of continuous selection and measurable selection. 
To do this, for any correspondence $\Phi:T\times Z\to 2^Y$, we first define 
\[
U_\Phi=\{(t, x)\in T\times Z: \Phi(t,x)\neq \emptyset\}.
\] 
For any $t\in T$, let $U_\Phi^t=\{x\in Z : (t, x)\in U_\Phi\}$ denote the $t$-section of $\Phi$,
and for any $x\in Z$, let $U_\Phi^x=\{t\in T:(t, x)\in U_\Phi\}$ denote the $x$-section of $\Phi$.
Let $Z$ be a topological space and $\Phi: T \times Z 
\to 2^Y$ be a nonempty valued correspondence. A function $\varphi: T \times Z \to Y$ is said to be a \emph{Carath\'eodory-type selection} from $\Phi$ if $\varphi(t,z) \in \Phi(t,z)$ 
for all $(t,z) \in U_\Phi$, $\varphi(\cdot,z):U_\Phi^z\to Y$ is measurable for all $z \in Z$, and $\varphi(t,\cdot):U_\Phi^t\to Y$ is continuous for all $t \in T$.

\medskip
If $B$ is a closed, convex subset of a normed linear space, then a \emph{supporting} set of $B$ is a closed 
convex subset $S$ of $B$, $S \neq B$, such that if an interior point of a segment in $B$ is in $S$, then the whole segment is in $S$. The set of 
all elements of $B$ which are not in any supporting set of $B$ will be denoted by $\mathbb I(B)$. The following facts below are due to 
Michael \cite{Michael}.

\begin{lemma}\label{lem1}
If a convex subset $B$ of $Y$ is either closed or has an interior point or is of finite-dimensional, then $\mathbb I({\rm cl} B) \subset B$.
\end{lemma}

\begin{lemma}\label{lem2}
Let $B$ be a nonempty, closed, convex separable subset of a Branch space $Y$, and $\{y_{i}:i=1,2,\cdots\}$ be a dense subset of $B$. If 
\begin{equation*}
    z_{i} = y_{i} + \frac{(y_{i}-y_{1})}{\max\{1,||y_{i}-y_{1}||\}} \mbox{ for all } i\ge 1 \mbox{ and } z =\sum^{\infty}_{i=1} \frac{1}{2^i} z_{i},
\end{equation*}
\\then $z\in \mathbb I(B)$.
\end{lemma}

\section{Continuous inclusion property}\label{sec:model}

The following Carath\'eodory-type selection theorem is due to KPY \cite{KPY:87}. 
\begin{theorem}\label{thm:KPY}
Let $(T,\mathscr T,\mu)$ be a complete finite measure space, $Y$ be a separable Banach space, and $Z$ be a complete, 
separable metric space. Suppose that $\Psi:T\times Z\to 2^Y$ be a convex $($possibly empty-$)$ valued 
correspondence such that $\Psi$ is jointly lower measurable and $\Psi(t,\cdot)$ is lower-semicontinuous for all
$t\in T$.  Furthermore, any one of the following conditions is true: $Y$ is finite-dimensional; $\Psi$ is closed-valued; and 
${\rm int}\Psi(t,x)\neq \emptyset$ for all $(t,x)\in U_\Psi$. Then there exists a Carath\'eodory-type selection 
$\psi:U_\Psi\to Y$ of $\restr{\Psi}{U_\Psi}$. 
\end{theorem}

Below we show by an example that lower-semicontinuity of $\Psi(t,\cdot)$ is not necessary to guarantee a 
Carath\'eodory-type selection. 
\begin{example}\label{exm:Psi}
Conider a Lebesgue measure space $([0, 1], \mathcal M, m)$ and define a correpondence $\Psi:[0, 1]\times \mathbb R\to \mathbb R$
such that 
 \[
  \Psi(t,x) = \left\{
  \begin{array}{ll}
  [0, 1],& \mbox{if $t\in [0, 1]$ and $x=0$;}\\[0.5em]
  \{0\}, & \mbox{otherwise.}
  \end{array}
  \right.
  \]
Notice that $\Psi(t,\cdot)$ is not lower-semicontinuous for all $t\in [0, 1]$. However, it admits a 
Carath\'eodory-type selection $\psi:[0, 1]\times \mathbb R\to \mathbb R$, defined by $\psi(t,x)=0$,
for all $(t,x)\in [0, 1]\times \mathbb R$. 
\end{example}
Motivated from the above example and also from applications point of view (see Section \ref{sec:applications}), 
we introduce the following notion of ``continuous inclusion property"
to establish a Carath\'eodory-type selection theorem for a correspondence which is not necessarily lower-semicontinuous.

\begin{definition}\label{CIPdef}
Let $(T,\mathscr T,\mu)$ be a measure space, $Y$ be a Banach space, and $Z$ be a metric space. A correspondence 
$\Psi:T\times Z\to 2^Y$ is said to have the \emph{continuous inclusion property} if for each $z\in Z$, there 
exists a correspondence $F_z:T\times Z\to 2^Y$ satisfying the following:
\begin{itemize}
\item[(i)] If $U_\Psi^z\neq \emptyset$ then there exists a collection $\{O_z^t:t\in U_\Psi^z\}$ of open neighbourhoods of $z$ in 
$Z$ such that $F_z(t,x)\neq \emptyset$ and $F_z(t,x)\subseteq \Psi(t, x)$ for all $x\in O_z^t$ and all $t\in U_\Psi^z$; and 
\item[(ii)] ${\rm con} F_z(t,\cdot): O_z^t\to 2^Y$ is lower-semicontinuous for all $t\in 
U_\Psi^z$ and ${\rm con}F_z(t,\cdot): Z\to 2^Y$ is lower-semicontinuous for all $t\notin 
U_\Psi^z$.\footnote{Note that one may replace the lower-semicontinuity of $F_z(t,\cdot): O_z^t\to 2^Y$ with that 
of $F_z(t,\cdot): Z\to 2^Y$ as the later implies that former.} 
\end{itemize}
\end{definition}

\begin{definition}\label{SCIPdef}
Let $(T,\mathscr T,\mu)$ be a measure space, $Y$ be a Banach space, and $Z$ be a metric space. A continuous 
inclusion property for a correspondence $\Psi:T\times Z\to 2^Y$ will be termed as the \emph{strong continuous 
inclusion property} if ${\rm con}F_z:T\times Z\to 2^Y$ is jointly lower measurable and any of the following conditions are satisfied:

\begin{itemize}

\item[(a)] $F_y=F_z$ for all $y,z\in Z$.

\item[(b)] The collection $\{F_z:z\in Z\}$ and $\{O_z^t:(t, z)\in U_\Psi\}$ are countable,
and the set $\{(t,x):x\in O_z^t\}$ is $\mathscr T\otimes \mathscr B(Z)$-measurable for 
each $z\in Z$,  where $\mathscr{B}(Z)$ is the 
Borel $\sigma$-algebra for the metric topology on $Z$.

\item[(c)] $U_\Psi\in \mathscr T\otimes \mathscr B(Z)$. For each $z\in Z$, the correspondence 
$\mathbb I:T\times Z\to 2^Z$, defined by 
$\mathbb I(t,x)=\{z\in U_\Psi^t:x\in O_z^t\}$, is jointly lower measurable (has measurable graph) and is finite valued.\footnote{The finiteness assumption 
is also follows from the assumption that $\{O_y^t:y\in Z\}$ is locally finite for each $t\in T$.}  Furthermore, for each fixed $(t,x)\in T\times Z$, 
the correpondence $H:Z\to 2^Y$, defined by $H(z)={\rm con }F_z(t, x)$, is continuous and such that 
$H(z)\subseteq B$ for all $z\in Z$ for some compact subset $B$ of $Y$. 

\end{itemize}
\end{definition}

\begin{remark}\label{rem:lower}
Note that if a correpondence $\Psi:T\times Z\to 2^Y$ is jointly lower measurable and the correspondence 
$\Psi(t,\cdot):Z\to 2^Y$ is lower-semicontinuous for all $t\in T$ then $\Psi$ satisfies the strong continuous inclusion 
property. Indeed, we can take
$O_z^t=U_\Psi^t$ for all $t\in U_\Psi^z$ if $U_\Psi^z\neq \emptyset$ and $F_z=\Psi$ for all $z\in Z$. Since $F_y=F_z$ for all $y,z\in Z$, the measurability 
condition is vacuously satisfied. The correspondence  ${\rm con}F_z(t,\cdot): O_z^t\to 2^Y$ is lower-semicontinuous (resp. has open graph) for all $t\in 
U_\Psi^z$ follows from the fact that $\Psi(t,\cdot):Z\to 2^Y$ is lower-semicontinuous (resp. has open graph) for all $t\in T$ and 
$U_\Psi^t=\{z\in Z:\Psi(t,z)\cap Y\neq \emptyset\}$ is weakly open in $Z$. The
rest of the conditions are also immediate.  
\end{remark}

\begin{remark}
The correspondence $\Psi$ in Example \ref{exm:Psi} trivially satisfies the strong continuous inclusion property. For each $z\in \mathbb R$,
we can define $O_z^t=\mathbb R$ and $F_z(t,x)=\{0\}$ for all $(t,x)\in U_\Psi=[0, 1]\times \mathbb R$. In fact, ${\rm con}F_z(t,\cdot):\mathbb 
R\to 2^{\mathbb R}$ is lower-semicontinuous for all $t\in T$; ${\rm con}F_z$ is jointly lower measurable; and $F_y=F_z$ for all $y,z\in Z$. 
\end{remark}

\begin{remark}
One can replace the lower-semicontinuity of ${\rm con} F_z(t,\cdot)$
with that of $F_z(t,\cdot)$, as convex hull of any  lower-semicontinuous correpondence is  lower-semicontinuous.
One can arrive at a similar conclusion for joint lower measurability. 
It is well-known that the mapping $H:\mathscr 
P_0(Y)\to \mathscr P_0(Y)$, defined by 
$H(B)={\rm con}(B)$ for all $B\in \mathscr P_0(Y)$, is continuous with respect to the lower 
topology. Given this fact and the fact that 
$F:T\times Z\to 2^Y$ is jointly lower measurable, we can conclude that ${\rm con}F_z:T\times Z\to 2^Y$ 
is jointly lower measurable. 
\end{remark}

\begin{remark}
The definition of continuous inclusion property is an extension of that in He-Yannelis \cite{HY:2017} to the case of a correspondence 
having two variables domain, in which one variable has a measure-theoretic structure whereas the other variable is endowed with a topological structure.      
Apart of this, in our definition the closed graph condition of 
${\rm con}F_z(t,\cdot)$ in He-Yannelis \cite{HY:2017} is replaced with the lower-semicontinuity of ${\rm con}F_z(t,\cdot)$. In the next section, we 
use this condition to establish Carath\'eodory-type selection and fixed-point theorems when the measure space is purely atomic. However, 
to deal with more general case, i.e, for an arbitrary complete finite positive measure space, we require the strong continuous 
inclusion property, which satisfies some measurability condition along with the continuous inclusion property.   
\end{remark}

\section{The main lemmata}\label{sec:mainlemmata}

Before we state and prove our main theorems (refer to Section \ref{sec:main}) we will need some preparatory results.  To this end, 
we use the following notation throughout the rest of the paper: define an operator $\mathbb K:\mathcal C_0\to 2^{\mathcal C}$ such that\footnote{Here,
$\mathcal C$ is the set of correspondences from $T\times Z$ into $Y$ and $\mathcal C_0$ is a sub-collection of 
$\mathcal C$ satisfying the (strong) continuous inclusion property.} 
\[
\mathbb K(\Psi)(t,x)=\bigcup\{{\rm int}F_z(t,x): x\in O_z^t \mbox{ and } \Psi(t,z)\neq \emptyset\},
\] 
where $\{F_z:z\in Z\}$ and $\{O_z^t:\Psi(t,z)\neq \emptyset\}$ are defined in Definition \ref{CIPdef}. Note that 
$\mathbb K(\Psi)(t,x)\neq \emptyset$ if and only if ${\rm int}F_z(t,x)\neq \emptyset$ for some 
$z\in Z$ for which $\Psi(t,z)\neq \emptyset$ and $x\in O_z^t$. In a special case when $\Psi$ is jointy lower measurable and 
$\Psi(t,\cdot)$ is lower-semicontinuous for all $t\in T$ then ${\rm int}\Psi(t,x)\neq \emptyset$ implies that 
$\mathbb K(\Psi)(t,x)\neq \emptyset$ (see Remark \ref{rem:lower}).  

In the next two lemmas, given  a correspondence $\Psi:T\times Z\to 2^Y$ satisfying the (strong) continuous inclusion 
property we construct a sub-correspondence $\Phi$ of a correspondence $\Psi$ by gluing a sub-collection $\{F_z:z\in Z\}$
in some way so that $\Phi$ should satisfy measurability and continuity conditions of Theorem \ref{thm:KPY} and furthermore,
$U_\Phi=U_\Psi$. This along with Theorem \ref{thm:KPY} will guarantee the existence of a  Carath\'eodory-type selection (see 
Section \ref{sec:main}). The main hurdles in establishing these lemmas are to define suitable sub-correspondences and the 
measurability of these sub-correspondences as as an abritrary union of lower measurable correspondences 
may not be measurable (see Lemma \ref{extcara:aux2}).

\begin{lemma}\label{extcara:aux1}
Let $(T,\mathscr T,\mu)$ be a complete finite atomic measure space, $Y$ be a separable Banach space, and $Z$ be a 
separable metric space. Suppose that $\Psi:T\times Z\to 2^Y$ be a convex (possibly empty-) valued 
correspondence such that $\Psi$ satisfies the continuous inclusion property. Then there exists a convex valued correspondence 
$\Phi:T\times Z\to 2^Y$ satisfying the following:
\begin{itemize}
\item[\emph{(A)}] $\Phi(t,x)\subseteq \Psi(t,x)$ for all $(t, x)\in U_\Psi$;
\item[\emph{(B)}] $U_\Psi= U_\Phi$;
\item[\emph{(C)}] $\Phi(t,\cdot):Z\to 2^Y$ is lower-semicontinuous for all $t\in T$; 
\item[\emph{(D)}] $\Phi$ is jointly lower measurable; and 
\item[\emph{(E)}] ${\rm int}\Phi(t,x)\neq \emptyset$ if $\mathbb K(\Psi)(t,x)\neq \emptyset$. 
\end{itemize}
\end{lemma}

\begin{proof}
Since the correspondence $\Psi$ satisfies the continuous inclusion property, we have that for each $z\in Z$, there exists a 
correspondence $F_z:T\times Z\to 2^Y$ satisfying:
\begin{itemize}
\item[(i)] If $U_\Psi^z\neq \emptyset$ then there exists a collection $\{O_z^t:t\in U_\Psi^z\}$ of open neighbourhoods of $z$ in $Z$
such that $F_z(t,x)\neq \emptyset$ and $F_z(t,x)\subseteq \Psi(t, x)$ for all $x\in O_z^t$ and all $t\in U_\Psi^z$; and 
 
\item[(ii)] ${\rm con} F_z(t,\cdot): O_z^t\to 2^Y$ is lower-semicontinuous for all $t\in 
U_\Psi^z$ and ${\rm con}F_z(t,\cdot): Z\to 2^Y$ is lower-semicontinuous for all $t\notin 
U_\Psi^z$.
\end{itemize}
The collection of open neighbourhoods defined above is therefore 
\[
\mathcal O=\left\{O_z^t:t\in U_\Psi^z \mbox{ and } z\in Z\right\}.
\]
Then $\mathcal O$ can be rewritten as $\mathcal O=\{O_z^t:(t,z)\in U_\Psi\}$. From \textcolor{red}{(i)}, it follows that $O_z^t\subseteq U_\Psi^t$ for all 
$z\in U_\Psi^t$. This implies that $U_\Psi^t$ is an open subset of $Z$ and  $\{O_z^t:z \in 
U_\psi^t\}$ forms an open cover of $U_\Psi^t$ for all $t\in T$. For each $(t, x)\in T\times Z$, we define 
\[
\mathbb J(t,x)=\left\{z\in Z: x\in O^t_z\right\}.
\]
Note that $\mathbb J(t,x)$ is (possibly empty) finite for all $(t,x)\in T\times Z$. Moreover, $\mathbb J(t,x)\neq \emptyset$
 if and only if $x\in U_\Psi^t$. Let $\Phi:T\times Z\to 2^Y$ be a 
correspondence such that 
 \[
  \Phi(t,x) = \left\{
  \begin{array}{ll}
  \Theta(t,x),& \mbox{if $\mathbb J(t, x)\neq \emptyset$;}\\[0.5em]
  \emptyset, & \mbox{otherwise,}
  \end{array}
  \right.
  \]
where $\Theta:T\times Z\to 2^Y$ is defined by
\[
 \Theta(t,x)={\rm con}\left(\bigcup\left\{{\rm con}F_{z}(t,x): z\in \mathbb J(t,x)\right\}\right).
\]
To verify Condition \textcolor{red}{(A)}, take an element $(t,x)\in U_\Psi$.
Since $x\in U_\Psi^t$, we have $\mathbb J(t,x)\neq \emptyset$, and therefore, 
$x\in O_{z}^t$ for all
$z\in \mathbb J(t,x)$. From \textcolor{red}{(i)} it follows that  
$F_{z}(t,x)\subseteq \Psi(t,x)$
for all $z\in \mathbb J(t,x)$. The convexity of $\Psi(t,x)$ yields that $\Phi(t,x)\subseteq \Psi(t,x)$. To verify Condition \textcolor{red}{(B)}, 
we first note that, for each 
$(t,x)\in U_\Psi$, we have $\mathbb J(t, x)\neq \emptyset$. Since $x\in O_{z}^t$ for all
$z\in \mathbb J(t,x)$, from \textcolor{red}{(i)} it follows that
$F_{z}(t,x)\neq \emptyset$ for all $z\in \mathbb J(t,x)$. This implies that $\Phi(t,x)\neq \emptyset$, which yields that 
$U_\Psi\subseteq U_\Phi$. Let $(t,x)\in U_\Phi$. From the definition of $\Phi$, it follows that $\mathbb J(t,x)\neq \emptyset$. This implies that 
$(t,x)\in U_\Psi$. Therefore, $U_\Phi\subseteq U_\Psi$. Thus, $U_\Psi=U_\Phi$, which verifies Condition \textcolor{red}{(B)}. 
In order to investigate Condition \textcolor{red}{(C)} and Condition \textcolor{red}{(D)}, we take an arbitrary open set $W$ in $Y$.
For each $t\in T$, we define 
\[
W_{t}=\{x\in Z:\Phi(t,x)\cap W\neq \emptyset\}.
\] 
From the definition of $\Phi$, it follows that $W_{t}= U_\Psi^{t}\cap \{x\in Z:\Theta(t,x)\cap W\neq \emptyset\}$. In view of Theorem \textcolor{red}{17.27} of Aliprantis-Border \cite{AB} and
the fact that the convex hull of a lower-semicontinuous correspondence is lower-semicontinuous, we conclude that $\Theta(t,\cdot)$ is 
lower-semicontinuous. Therefore, $\{x\in Z:\Theta(t,x)\cap W\neq \emptyset\}$ 
is open in $Z$. Since $U_\Psi^{t}$ is open in $Z$, we have that $W_{t}$ is open in $Z$. Consequently, $\Phi(t,\cdot)$ is lower-semicontinuous for all 
$t\in T$.
To verify the joint lower measurability of $\Phi$, let $\mathbb B=\{t\in T:W_{t}\neq \emptyset\}$. Since $\mu(T)<\infty$ and $(T, \mathscr T, \mu)$ 
is atomic, we have that $T$ is countable and thus, $\mathbb B$ is countable. Consequently,
\[
\{(t,x)\in T\times Z:\Phi(t,x)\cap W\neq\emptyset\}=\bigcup\{(\{t\}\times W_{t}):t\in \mathbb B\}.
\]
Since each atom $\{t\}$ belongs to $\mathscr T$ and each $W_t$ is open in $Z$, we conclude that 
$\{(t,x)\in T\times Z:\Phi(t,x)\cap W\neq\emptyset\}$ is $\mathscr T\otimes \mathscr B(Z)$-measurable. Therefore,
$\Phi$ is jointly lower measurable. Finally, under the interiority condition and the definition of $\Phi$, we conclude that 
Condition \textcolor{red}{(E)} is verified.               
\end{proof}

\begin{lemma}\label{extcara:aux2}
Let $(T,\mathscr T,\mu)$ be a complete finite measure space, $Y$ be a separable Banach space, and $Z$ be a 
separable metric space. Suppose that $\Psi:T\times Z\to 2^Y$ be a convex $($possibly empty-$)$ valued correspondence 
such that $\Psi$ satisfies the strong continuous inclusion property. Then there exists a convex valued correspondence 
$\Phi:T\times Z\to 2^Y$ satisfying the following:
\begin{itemize}
\item[\emph{(A)}] $\Phi(t,x)\subseteq \Psi(t,x)$ for all $(t, x)\in U_\Psi$;
\item[\emph{(B)}] $U_\Psi= U_\Phi$;
\item[\emph{(C)}] $\Phi(t,\cdot):Z\to 2^Y$ is lower-semicontinuous for all $t\in T$; 
\item[\emph{(D)}] $\Phi$ is jointly lower measurable; and 
\item[\emph{(E)}] ${\rm int}\Phi(t,x)\neq \emptyset$ if $\mathbb K(\Psi)(t,x)\neq \emptyset$. 
\end{itemize}
\end{lemma}

\begin{proof}Since the correspondence $\Psi$ satisfies the strong
continuous inclusion property, we have that for each $z\in Z$, there exists a correspondence $F_z:T\times Z\to 2^Y$ 
such that the following:
\begin{itemize}
\item[(i)] If $U_\Psi^z\neq \emptyset$ then there exists a collection $\{O_z^t:t\in U_\Psi^z\}$ of open neighbourhoods of $z$ in $Z$
such that $F_z(t,x)\neq \emptyset$ and $F_z(t,x)\subseteq \Psi(t, x)$ for all $x\in O_z^t$ and all $t\in U_\Psi^z$;
 
\item[(ii)] ${\rm con} F_z(t,\cdot): O_z^t\to 2^Y$ is lower-semicontinuous for all $t\in 
U_\Psi^z$ and ${\rm con}F_z(t,\cdot): Z\to 2^Y$ is lower-semicontinuous for all $t\notin 
U_\Psi^z$; and 
\item[(iii)] ${\rm con}F_z:T\times Z\to 2^Y$ is jointly lower measurable.
\end{itemize}
 In addition, any one of the following conditions holds:
\begin{itemize}
\item[(a)] $F_y= F_z$ for all $y,z\in Z$.
\item[(b)] The collection $\{F_z:z\in Z\}$ and $\{O_z^t:(t, z)\in U_\Psi\}$ are countable,
and the set $\{(t,x):x\in O_z^t\}$ is $\mathscr T\otimes \mathscr B(Z)$-measurable for all $z\in Z$.

\item[(c)] $U_\Psi$ is $\mathscr T\otimes \mathscr B (Z)$-measurable. The correspondence $\mathbb I:T\times Z\to 2^Z$, defined by 
$\mathbb I(t,x)=\{z\in U_\Psi^t:x\in O_z^t\}$, is jointly lower measurable and is finite-valued.
Furthermore, for each fixed $(t,x)\in T\times Z$, the correspondence 
$H:Z\to 2^Y$, defined by $H(z)={\rm con}F_z(t, x)$, is continuous and such that 
$H(z)\subseteq B$ for all $z\in Z$ for some compact subset $B$ of $Y$. 
\end{itemize}
Define $\mathcal A=\{F_z:z\in Z\}$ and $\mathcal O=\{O_z^t:(t,z)\in U_\Psi\}$. The rest of the proof is decomposed into the following three cases:

\medskip
{\bf Case \textcolor{red}{(a)}:} Since $F_y=F_z$ for all $y,z\in Z$, we denote this common value of $F_z$ for all $z\in Z$
by the correspondence $F$.  Let $\Phi:T\times Z\to 2^Y$ be a correspondence 
such that $\Phi(t,x) = {\rm con}F(t,x)$ for all $(t,x)\in T\times Z$. Since $\Psi$ is convex-valued and $F(t,x)=F_x(t,x)$, we conclude that 
\textcolor{red}{(i)} implies Condition \textcolor{red}{(A)}. To verify Condition 
\textcolor{red}{(B)}, we note that, for each 
$(t,x)\in U_\Psi$, we have $x\in O_x^t$ and $F_x(t,x)\neq \emptyset$. 
This implies that $\Phi(t,x)\neq \emptyset$, which yields that $U_\Psi\subseteq U_\Phi$. Let $(t,x)\in U_\Phi$. From the 
definition of $\Phi$, it follows that $F(t,x)\neq \emptyset$. Since $F(t,x)=F_x(t,x)$, from \textcolor{red}{(i)} it follows that $(t,x)\in U_\Psi$.  
Consequently, $U_\Phi\subseteq U_\Psi$. 
Thus, $U_\Psi=U_\Phi$, which verifies Condition \textcolor{red}{(B)}. Since ${\rm con}F(t,\cdot)={\rm con}F_y(t,\cdot)$ is lower-semicontinuous on $O_z^t$ 
for all $t\in U_\Psi^z$ and all $z\in Z$, and $U_\Psi^t=\bigcup\{O_z^t:y\in U_\Psi^t\}$, we conclude that ${\rm con}F(t,\cdot):U_\Psi^t
\to 2^Y$ is lower-semicontinuous if $U_\Psi^t\neq \emptyset$. From Condition \textcolor{red}{(B)} and the definition of 
$\Phi$, it follows that $U_F=U_\Phi=U_\Psi$. Thus, $F(t,\cdot):Z\to 2^Y$ is 
lower-semicontinuous for all $t\in T$. This verifies Condition \textcolor{red}{(C)}. 
From \textcolor{red}{(iii)}, we conclude that $\Phi$ is jointly 
lower measurable. This verifies Condition \textcolor{red}{(D)}. Finally, under the interiority condition and the definition of $\Phi$, we conclude that 
Condition \textcolor{red}{(E)} is verified.          

\medskip
{\bf Case \textcolor{red}{(b)}:} Let $\Phi:T\times Z\to 2^Y$ be a correspondence 
such that 
 \[
  \Phi(t,x) = {\rm con}\left(\bigcup\left\{{\rm con}\widetilde{F_{z}}(t,x): z\in Z\right\}\right),
  \]
where the correspondence $\widetilde{F_{z}}:T\times Z\to 2^Y$ is defined as 
\[
  \widetilde{F_{z}}(t, x) = \left\{
  \begin{array}{ll}
  F_z(t,x),& \mbox{if $x\in O_z^t$;}\\[0.5em]
  \emptyset, & \mbox{otherwise.}
  \end{array}
  \right.
  \]
Since $\widetilde{F_{z}}(t, x)\subseteq \Psi(t,x)$ for each $(t,x)\in U_\Psi$, by the convexity of $\Psi(t, x)$, 
we have $\Phi(t, x)\subseteq \Psi(t, x)$ for each $(t, x)\in U_\Psi$. This verifies Condition \textcolor{red}{(A)}. On 
the other hand, for each $(t,x)\in U_\Psi$, we have $x\in O_x^t$ and $F_x(t,x)\neq \emptyset$. 
This implies that $\Phi(t,x)\neq \emptyset$, which yields that $U_\Psi\subseteq U_\Phi$. Let $(t,x)\in U_\Phi$. From the 
definition of $\Phi$, it follows that $F_{z_0}(t,x)\neq \emptyset$ for some $z_0\in Z$ with $x\in O_{z_0}^{t}$. As 
$O_{z_0}^t\in \mathcal O$, we conclude that $t\in U_\Psi^{z_0}$. From \textcolor{red}{(i)}, 
it follows that $F_{z_0}(t,x)\subseteq \Psi(t,x)$ and hence, $\Psi(t, x)\neq \emptyset$.
Consequently, $U_\Phi\subseteq U_\Psi$. Thus, $U_\Psi=U_\Phi$, which verifies Condition \textcolor{red}{(B)}. 
Furthermore, in view of Theorem \textcolor{red}{17.27} of Aliprantis-Border \cite{AB}, we conlcude that $\Phi(t,\cdot)$ is lower-semicontinuous for all $t\in T$.\footnote{Note that 
the lower-semicontinuity of $F_z$ implies that $\widetilde{F_z}(t,\cdot):Z\rightrightarrows Y$ is lower-semicontinuous for all $t\in T$.} 
This verifies Condition \textcolor{red}{(C)}. 
Since $\mathcal A$ and $\mathcal O$ are countable, the 
collection $\{\widetilde{F_y}:y\in Z\}$ is also countable. We denote such a collection by $=\{\widetilde{F_1},\widetilde{F_2},\cdots\}$. Thus, 
\[
 \Phi(t,x) =  {\rm con}\left(\bigcup\left\{{\rm con}\widetilde{F_{i}}(t,x): i\in \mathbb N\right\}\right).
\]
Below we show that each $\widetilde{F_i}$ is jointly lower measurable. To this end, consider an arbitrary correspondence $\widetilde{F_z}$ and choose a nonempty open set $W$ in $Y$. Since
\begin{equation*}
\begin{split}
\left\{(t,x): \widetilde{F_z}(t,x) \cap W \neq \emptyset\right\} &= \left\{(t,x): x\in O_z^t\right\}\cap \left\{(t,x):F_z(t,x) \cap W \neq \emptyset\right\},
\end{split}
\end{equation*}
$\widetilde{F_z}$ is jointly lower measurable. This is due to the fact that $\left\{(t,x): x\in O_z^t\right\}$ is $\mathscr T\otimes \mathscr B(Z)$-measurable  
and $\left\{(t,x):F_z(t,x) \cap W \neq \emptyset\right\}$ is $\mathscr T\otimes \mathscr B(Z)$-measurable by the joint lower measurability of $F_z$. By 
Lemma \textcolor{red}{18.4} of Aliprantis-Border \cite{AB}, we conclude that $\Phi$ is jointly lower measurable. This verifies Condition \textcolor{red}{(D)}. 
Finally, under the interiority condition and the definition of $\Phi$, we conclude that 
Condition \textcolor{red}{(E)} is verified.

\medskip
{\bf Case \textcolor{red}{(c)}:} Let $\Phi:T\times Z\to 2^Y$ be a correspondence 
such that 
 \[
  \Phi(t,x) = \left\{
  \begin{array}{ll}
  {\rm con}\left(\bigcup\left\{{\rm con}F_{z}(t,x): z\in \mathbb I(t,x)\right\}\right),& \mbox{if $\mathbb I(t,x)\neq \emptyset$;}\\[0.5em]
  \emptyset, & \mbox{otherwise.}
  \end{array}
  \right.
  \]
As in {\bf Case 1}, one can verify Conditions \textcolor{red}{
(A)}-\textcolor{red}{(C)} and \textcolor{red}{(E)} for the correspondence $\Phi$. In what follows, we show that 
$\Phi$ is jointly lower measurable. To this end, we take an arbitrary open set $W$ in $Y$ and define a  
correspondence $\Theta:T\times Z\to 2^Y$ such that  
\[
\Theta(t,x)=\bigcup\left\{{\rm con}F_{z}(t,x): z\in \mathbb I(t,x)\right\}.
\]
We have to verify that 
\[
B=\{(t,x)\in T\times Z:\Phi(t, x)\cap W\neq \emptyset\}\in \mathscr T\otimes \mathscr B(Z).
\]
Notice that $B$ can be written as 
\[
B=\{(t,x)\in T\times Z:\mathbb I(t, x)\neq \emptyset\}\cap \{(t,x)\in T\times Z:{\rm con}\Theta(t,x)\cap W\neq \emptyset\}. 
\]
To verify the measurability of $\{(t,x)\in T\times Z:\mathbb I(t, x)\neq \emptyset\}$, we first show that 
\begin{eqnarray}\label{eqn:psi}
\{(t,x)\in T\times Z:\mathbb I(t, x)\neq \emptyset\}=\left\{(t,x)\in T\times Z:\Psi(t, x)\neq \emptyset\right\}.
\end{eqnarray}
Let $\mathbb I (t,x)\neq \emptyset$. Choose an element $z\in \mathbb I(t,x)$. From the definition $\mathbb I(t,x)$, 
it follows that $z\in U_\Psi^t$ and $x\in O_z^t$. This implies that $F_z(t,x)\neq \emptyset$ and $F_z(t,x)\subseteq 
\Psi(t,x)$. Hence, $\Psi(t,x)\neq \emptyset$. Conversely, let $\Psi(t,x)\neq \emptyset$. Then $x\in U_\Psi^t$ and 
$x\in O_x^t$. Therefore, $x\in \mathbb I(t,x)$, which implies that $\mathbb I(t,x)\neq \emptyset$. Thus, Equation 
(\ref{eqn:psi}) is verified. Equation \ref{eqn:psi}) along with the fact that $U_\Psi\in \mathscr T\otimes \mathscr B(Z)$ 
together imply that $\{(t,x)\in T\times Z:\mathbb I(t, x)\neq \emptyset\}$ is $\mathscr T\otimes \mathscr B(Z)$-measurable. 
To show the measurability of $\{(t,x)\in T\times Z:{\rm con}\Theta(t,x)\cap W\neq \emptyset\}$, it just enough to verify the 
joint lower measurability of $\Theta$.\footnote{Since the convex hull of a lower measurable correspondence is lower measurable.} 
To this end, first note that $\mathscr K_0(Z)$ equipped with Housdorff metric topology is a polish space and by Corollary \textcolor{red}{3.95} of Aliprantis-Border \cite{AB}, it follows that a sequence $\{C_n:n\ge 1\}\subseteq \mathscr K_0(Z)$ converges to $C\in 
\mathscr K_0(Z)$ with respect to the Hausdorff metric if and only if $ {\rm lim\, sup} \, C_n={\rm lim\, inf}\, C_n= C$. 
Define a correspondence $Q:T\times Z\times \mathscr K_0(Z)\rightarrow 2^Y$ such that 
\[
Q(t,x,C)= \bigcup\left\{\overline{\rm con}F_{z}(t,x): z\in C\right\},
\]
and then define a correspondence $N:T\times Z\times \mathscr K_0(Z)\to 2^Y$ such that 
$N(t,x,C)={\rm cl} Q(t,x,C)$.

\medskip
\noindent
{\bf Showing the Hausdorff continuity of \textit{\textbf{N$($t,x,$\cdot$$)$}}:} Fix an element $(t,x)\in T\times Z$. By our hypothesis,
the correspondence $H:Z\to 2^Y$, defined by $H(z)={\rm con }F_z(t, x)$, is continuous and such that 
$H(z)\subseteq B$ for all $z\in Z$ for some compact subset $B$ of $Y$.  By Lemma \textcolor{red}{17.22} of Aliprantis-Border \cite{AB}, the correspondence 
$G:Z\to 2^Y$, defined by $G(z)= {\rm cl}H(z)$, is also continuous. Since $G$ is compact valued, 
by Theorem \textcolor{red}{3.91} and Corollary \textcolor{red}{3.95} of Aliprantis-Border \cite{AB}, it follows that, for any sequence $\{z_n:n\ge 1\}\subseteq Z$ converges 
to $z\in Z$, we have 
\[
{\rm Ls} \, G(z_n)={\rm Li} \, G(z_n)= G(z).
\]
Let $\{C_n:n\ge 1\}\subseteq \mathscr K_0(Z)$ be a sequence converging to $C\in \mathscr K_0(Z)$. We show that 
\[
{\rm Ls} \, Q(t,x,C_n)={\rm Li} \, Q(t,x,C_n)= Q(t,x,C).
\]
Since ${\rm Li} \, Q(t,x,C_n)\subseteq {\rm Ls} \, Q(t,x,C_n)$, it only remains to verify that 
\[
{\rm Ls} \, Q(t,x,C_n)\subseteq Q(t,x,C)\subseteq {\rm Li} \, Q(t,x,C_n).
\]
We first prove that $Q(t,x,C)\subseteq {\rm Li} \, Q(t,x,C_n)$. If $Q(t,x,C)=\emptyset$ then the inclusion trivially follows. So, 
we let $y \in  Q(t,x,C)$. Then there exists some $z\in C$ such that $y\in G(z)$. Since $C = {\rm Li}\, C_n$, there 
exists some sequence $\{z_n:n\ge 1\}$ such that $z_n \in C_n$ for all $n\ge 1$ and $\{z_n:n\ge 1\}$ converges 
to $z$. Therefore, from above, it follows that $G(z)={\rm Li} \, G(z_n)$, which yields that $y\in {\rm Li} \, G(z_n)$.  
Since $G(z_n)\subseteq Q(t,x,C_n)$, we conclude that $y\in {\rm Li} \, Q(t,x,C_n)$. Consequently, 
\[
Q(t,x,C)\subseteq {\rm Li} \, Q(t,x,C_n).
\]
We now prove that ${\rm Ls} \, Q(t,x,C_n)\subseteq Q(t,x,C)$. If ${\rm Ls} \, Q(t,x,C_n)=\emptyset$ then there 
is nothing to verify. So, we assume that ${\rm Ls} \, Q(t,x,C_n)\neq \emptyset$ and pick an element 
$y\in {\rm Ls} \, Q(t,x,C_n)$. Therefore, there exists a subsequence $\{n_k:k\ge 1\}$ of natural numbers such that 
$n_1 < n_2 < \cdots$ and a sequence $\left\{y_{k}:k\ge 1\right\}$ such that  $y_{k} \in Q(t,x,C_{n_k})$ for all $k\ge 1$ 
and $\left\{y_{k}:k\ge 1\right\}$ converges to $y$. Thus, we can find an element $z_{k} \in C_{n_k}$ such that $y_{k} \in 
G(z_{k})$ for all $k\ge 1$. Since $z_k\in C_{n_k}$ for all $k\ge 1$ and $\{H(C_{n_k}, C):k\ge 1\}$ converges to $0$, we 
conclude that $\{d(z_{k}, C):k\ge 1\}$ converges to $0$.\footnote{Recall that $H(A, B)$ denotes the Hausdorff 
distance between the sets $A$ and $B$.} Therefore, there exists an element  $z\in C$ such that $\{z_k:k\ge 1\}$
converges to $z$. Furthermore, $y\in G(z)$, which further implies that $y\in Q(t,x,C)$.  Consequently, 
\[
{\rm Ls} \, Q(t,x,C_n)\subseteq Q(t,x,C).
\]
Since ${\rm Ls} \, Q(t,x,C_n)= {\rm Ls} \, {\rm cl} Q(t,x,C_n)$,  ${\rm Li} \, Q(t,x,C_n)= {\rm Li} \, {\rm cl}Q(t,x,C_n)$
and these sets are closed, we have that $Q(t,x,C)$ is closed and 
\[
{\rm Ls} \, N(t,x,C_n)={\rm Li} \, N(t,x,C_n)= N(t,x,C).
\]
Since these sets are contained in the compact set $B$, by Corollary \textcolor{red}{3.95} of Aliprantis-Border \cite{AB}, we conclude that $\{N(t,x,C_n):n\ge 1\}$ converges to $N(t,x,C)$ in 
the Housdorff metric topology, which implies that the function $N(t,x,\cdot):\mathscr K_0(Z)\to \mathscr K_0(Y)$ is Hausdorff continuous.

\medskip
\noindent
\medskip
\noindent
{\bf Showing the measurability of \textit{\textbf{N$($$\cdot,\cdot$,C$)$}}:} To see the lower measurability of $N(\cdot,\cdot,C)$ 
for any set $C\in \mathscr K_0(Z)$, we first choose an arbitrary finite set $C=\{z_1,\cdots,z_m\}$ and 
an open set $V$ in $Y$. Since 
\begin{equation*}
\begin{split}
\left\{(t,x)\in T\times Z: N(t,x,C) \cap V \neq \emptyset\right\} &= \left\{(t,x)\in T\times Z: Q(t,x,C) \cap V \neq \emptyset\right\}\\
                                                                                  &= \left\{(t,x) \in T\times Z: \left[ \bigcup_{i = 1}^{m} \overline{\rm con}F_{z_i}(t,x)\right] \cap V \neq \emptyset \right\}\\
                                                                                  & = \bigcup_{i = 1}^{m}\left\{(t,x)\in T\times Z: \overline{\rm con} F_{z_i}(t,x) \cap V \neq \emptyset\right\}\\
                                                                                  & = \bigcup_{i = 1}^{m}\left\{(t,x)\in T\times Z: {\rm con}F_{z_i}(t,x) \cap V \neq \emptyset\right\},
\end{split}
\end{equation*}
by \textcolor{red}{(iii)}, we conclude that $N(\cdot,\cdot,C)$ is jointly lower measurable for any finite set $C$. Since $Z$ is a separable metric space, 
there is a countable dense set $D$ in $Z$. We denote by $\mathscr D$ the family of all finite subsets of $D$. By Corollary \textcolor{red}{3.90} and 
Theorem \textcolor{red}{3.91} of Aliprantis-Border \cite{AB}, $\mathscr D$ is dense in 
$(\mathscr K_0(Z),\tau_H)$. Let $C$ be an element of $\mathscr K_0(Z)$. Then there exists a sequence 
$\{D_n:n\ge 1\}\subseteq \mathscr D$ converging to $C$ in the Hausdorff metric. Since $N(t,x,\cdot):(\mathscr K_0(Z), \tau_H)\to
(\mathscr K_0(Y), \tau_H)$ is continuous for all $(t,x)\in T\times Z$, we have that $\{N(t,x,D_n):n\ge 1\}$ converges to 
$N(t,x,C)$. Since each $N(\cdot,\cdot,D_n):T\times Z\to (\mathscr K_0(Y), \tau_H)$ is jointly measurable function, we 
conclude that $N(\cdot,\cdot,C):T\times Z\to (\mathscr K_0(Y), \tau_H)$ is also jointly measurable function.

\medskip
Since $N(\cdot,\cdot, C):T\times Z\to (\mathscr K_0(Y), \tau_H)$ is a jointly measurable function for each $C\in \mathscr K_0(Z)$ 
and $N(t,x, \cdot):(\mathscr K_0(Z), \tau_H)\to (\mathscr K_0(Y), \tau_H)$ is a continuous function for each $(t,x)\in T\times Z$, we 
conclude that $N:T\times Z\times \mathscr K_0(Z)\to (\mathscr K_0(Y), \tau_H)$ is a Carath\'{e}odory
function. Since $\mathbb I:T\times Z\to \mathscr K_0(Z)$ is a finite-valued jointly measurable function, by Lemma \textcolor{red}{8.2.3} of Aubin and Frankowska \cite{AF}, we have that $(t,x)\mapsto N(t,x,\mathbb I(t,x))$ is a jointly measurable function. From the definition of 
$\Theta$, it follows that 
\[
{\rm cl} \Theta(t,x) = \bigcup\left\{\overline{\rm con}F_{z}(t,x): z\in \mathbb I(t,x)\right\}=Q(t,x,\mathbb I(t,x)).
\]
Consequently, $Q(t,x,\mathbb I(t,x))$ is a closed set and hence, ${\rm cl} \Theta(t,x)=N(t,x,\mathbb I(t,x))$. 
Consider a correpondence $\Upsilon:T\times Z\to 2^Y$ such that $\Upsilon(t,x)={\rm cl} \Theta(t,x)$,
for all $(t,x)\in T\times Z$. Then $\Upsilon:T\times Z\to \mathscr K_0(Y)$ is a jointly measurable function. Thus, 
the correspondence $\Upsilon:T\times Z\to 2^Y$ is jointly lower measurable and hence, $\Theta$ is jointly lower measurable. 
This completes the proof of Condition \textcolor{red}{(D)}. 
\end{proof}

\begin{lemma}\label{lem:upper}
    Let $(T,\mathscr T,\mu)$ be a complete measure space, and let $X$ and $Y$ be separable metric spaces. Suppose that $\Psi:T\times X\to 2^Y$ be a (possibly empty-valued) correspondence 
  such that $\Psi$ satisfies the continuous inclusion property if the measure space is atomic and the strong continuous 
inclusion property, otherwise. Let $\psi:U_\Psi \to Y$ be a Carath\'eodory-type selection for $\Psi$. Suppose further that $F:T\times X\to 2^Y$ is 
a correpondence. Then the correspondence $G:T \times X \to 2^{Y}$, defined by 
\[
  G(t,x) = \left\{
  \begin{array}{ll}
  \{\psi(t,x)\},& \mbox{if $(t,x)\in U_\Psi$;}\\[0.5em]
  F(t,x), & \mbox{otherwise,}
  \end{array}
  \right.
  \]
satisfies the following conditions:
\begin{enumerate}
\item[\emph{(i)}] If $F(t,\cdot)$ is upper-semicontinuous then $G(t,\cdot)$ is upper-semicontinuous.
\item[\emph{(ii)}] If $F(t,\cdot)$ is lower-semicontinuous then $G(t,\cdot)$ is lower-semicontinuous.
\item[\emph{(iii)}] If $F$ is jointly lower measurable then $G$ is jointly lower measurable.
\item[\emph{(iv)}] If $F(\cdot,x)$ is lower measurable then $G(\cdot,x)$ is lower measurable. 
\end{enumerate}
\end{lemma}

\begin{proof}
By Lemma \ref{extcara:aux1} and Lemma \ref{extcara:aux2}, we can find a correspondence 
$\Phi:T\times X\rightarrow 2^Y$ satisfying the following:
\begin{itemize}
\item[(A)] $\Phi(t,x)\subseteq \Psi(t,x)$ for all $(\omega, x)\in U_{\Psi}$;
\item[(B)] $U_{\Psi}= U_\Phi$;
\item[(C)] $\Phi(t,\cdot):Z\to 2^Y$ is lower-semicontinuous for all $t\in T$; and 
\item[(D)] $\Phi$ is jointly lower measurable.
\end{itemize}
By \textcolor{red}{(C)}, we have that 
\[
U_\Phi^t=\left\{x\in X:\Phi(t,x)\cap Y\neq \emptyset\right\}
\]
is an open subset  of $X$. If $F(t,\cdot)$ is upper-semicontinuous, by Lemma \textcolor{red}{6.1} of Yannelis-Prabhakar \cite{YD:83}, 
we conclude that $G(t,\cdot)$ is upper-semicontinuous for all $t\in T$. For any $x\in X$, we have  
\[
U_\Phi^x= {\rm proj}_T\left(\left\{(t,x)\in T\times X:\Phi(t,x)\cap Y\neq \emptyset\right\}\cap (T\times \{x\})\right).
\]
In view of \textcolor{red}{(D)} and the measurable projection theorem \cite{Yannelis:91a}, we have that $U_\Phi^x$ is measurable. 
Therefore, by \textcolor{red}{(B)}, we conclude that $U_{\Psi}^t$ is open in the relative topology of $X$ for all $t\in T$ and 
$U_{\Psi}^x$ is measurable for all $x\in X$. Since for each $t\in T$ the function $\psi(t,\cdot)$ is continuous 
on $U_{\Psi}^{t}$ and for each $x \in X$ the function $\psi(\cdot,x)$ is measurable on $U_{\Psi}^{x}$, the function 
$\psi$ is jointly measurable. To see the joint lower measurability of $G$, choose 
an open set $V$ in $Y$. Define
\[
A=\{(t,x) \in T \times X :G(t,x) \cap V \neq \emptyset\}=B \cup C,
\]
 where 
 \[
 B= \{(t,x) \in U_{\Psi}: \psi(t,x) \in V\}
 \]
  and 
  \[
  C =\{(t,x) \in (T \times X)\setminus U_{\Psi}: F(t,x)\cap V\neq \emptyset\}.
  \] 
  Clearly, $B, C \in \mathscr T \otimes \mathscr B(X)$ and therefore $A = B \cup C$ belongs to $ \mathscr T \otimes \mathscr B(X)$.
Consequently, $G$ is jointly lower measurable. The lower-semicontinuity of $G(t,\cdot)$ and the lower measurability of $G(\cdot,x)$
can be established analogous to the joint measurability of $G$. 
\end{proof}

\section{Main results}\label{sec:main}

In this section, we provide the proof of a Carath\'eodory-type selection theorem for correspondences satisfying the (strong) continuous inclusion 
property, which generalizes the continuous and measurable selections theorems simultaneously. By Remark \ref{rem:lower}, it follows that our 
result also generalizes the Carath\'eodory-type selection theorem of KPY \cite{KPY:87}.

\begin{theorem}\label{main:cara}
Let $(T,\mathscr T,\mu)$ be a complete finite measure space, $Y$ be a separable Banach space, and $Z$ be a complete, 
separable metric space. Suppose that $\Psi:T\times Z\to 2^Y$ be a convex $($possibly empty-$)$ valued 
correspondence such that $\Psi$ satisfies the continuous inclusion property if the measure space is atomic and the strong continuous 
inclusion property, otherwise.  Furthermore, any one of the following conditions is true: $Y$ is finite-dimensional; $\Psi$ is closed-valued; and 
$\mathbb K(\Psi)(t,x)\neq \emptyset$ for all $(t,x)\in U_\Psi$.
Then there exists a Carath\'eodory-type selection $\psi:U_\Psi\to Y$ of $\restr{\Psi}{U_\Psi}$. 
\end{theorem}

\begin{proof}
By Lemma \ref{extcara:aux1} and Lemma \ref{extcara:aux2} (see Section \ref{sec:mainlemmata}), we have that there exists a convex valued correspondence 
$\Phi:T\times Z\to 2^Y$ satisfying the following:
\begin{itemize}
\item[(A)] $\Phi(t,x)\subseteq \Psi(t,x)$ for all $(t, x)\in U_\Psi$;
\item[(B)] $U_\Psi= U_\Phi$;
\item[(C)] $\Phi(t,\cdot):Z\to 2^Y$ is lower-semicontinuous for all $t\in T$; 
\item[(D)] $\Phi$ is jointly lower measurable; and 
\item[(E)] ${\rm int}\Phi(t,x)\neq \emptyset$ for all $(t, x)\in U_\Psi$. 
\end{itemize}
Define a correspondence $F:T \times Z \to 2^Y$ such that $F(t,x)={\rm cl} \Phi(t,x).$ Since $\Phi(t,\cdot):Z\to 2^Y$ is 
lower-semicontinuous so is $F(t,\cdot)$ for all $t\in T$. Furthermore, $F$ is jointly lower measurable and $U_F= U_\Phi=U_\Psi$.
We conclude the proof by considering the following two cases.

\medskip
{\bf Case 1.} $\Psi$ {\it  is closed-valued}.
In this case, we have $F(t,x)\subseteq \Psi(t,x)$ for all $(t,x)\in U_\Psi$. Applying Theorem \textcolor{red}{3.1} of KPY \cite{KPY:87}, we can 
guarantee the existence of a Carath\'eodory-type selection $\psi:U_F\to Y$ of $\restr{F}{U_F}$. Thus, $\psi$ is a Carath\'eodory-type 
selection of $\restr{\Psi}{U_\Psi}$. 

\medskip
{\bf Case 2.} $\Psi$ {\it  is not necessarily closed-valued}.
 By Lemma \textcolor{red}{5.2} of KPY \cite{KPY:87}, there exists a sequence 
$\{\varphi_{k}(t,x):k\ge 1\}$ of Carath\'eodory-type selections that is dense in $F(t,x)$ for all $(t,x)\in U_\Phi$. For each $k\ge 1$, let  
\[
    \psi_{k}(t,x) = \varphi_{l}(t,x) +\frac{\varphi_{k}(t,x)-\varphi_{1}(t,x)}{\max\{1,\|\varphi_{k}(t,x)-\varphi_{1}(t,x)\|\}}.
\]
Define a function $\psi:U_\Phi\to Y$ such that
\[
  \psi(t,x)= \sum^{\infty}_{k=1} \frac{1}{2^{k}} \psi_{k}(t,x).
\]
By Lemma \ref{lem2}, $\psi(t,x)\in \mathbb I(F(t,x))$ for all $(t,x) \in U_\Phi$. Since the series defining $\psi$ converges uniformly, it follows that for each 
each $x\in X$, $\psi(\cdot,x)$ is measurable and for each $t \in T$, $\psi(t,\cdot)$ is continuous. By Lemma \ref{lem1}, $\psi(t,x) \in \mathbb 
I(F(t,x)) \subset \Phi(t,x)$. By virtue of \textcolor{red}{(A)} and \textcolor{red}{(B)},
 we have $\psi(t,x) \in \Psi(t,x)$ for all $(t,x)\in U_\Psi$. This completes the proof of the theorem.
\end{proof}

Let $(T,\mathscr T, \mu)$ be a measure space and $X$ be a nonempty subset of any linear topological space. A correspondence $\Psi:T\times X\to 2^X$ is 
said to have a \emph{random fixed point} if there exists a measurable function $x:T\to X$ such that $x(t)\in \Psi(t,x(t))$ $\mu$-a.e. on $T$. Below we provide a random fixed point theorem. %%This result generalizes a theorem of Bohnenblust and Karlin . For other random fixed point results see[16,23].
\begin{theorem}\label{thm:fixed}
Let $(T,\tau,\mu)$ be a complete finite measure space, and $X$ be a nonempty compact convex subset of a separable Branch space $Y$. Let 
$\Psi:T \times X \to 2^X$ be a nonempty, convex valued correspondence such that $\Psi$ satisfies the continuous inclusion property 
 if the measure space is atomic and the strong continuous inclusion property, otherwise. Furthermore, any one of the following conditions is true: 
$Y$ is finite-dimensional; $\Psi$ is closed-valued; and 
$\mathbb K(\Psi)(t,x)\neq \emptyset$ for all $(t,x)\in U_\Psi$.
Then $\Psi$ has a random fixed point.
\end{theorem}

\begin{proof}
By Theorem \ref{main:cara}, it follows that there exists a function $\psi: T \times X \to X $ such that $\psi(t,x) \in \Psi(t,x)$ for all $(t,x) \in T \times X$, and 
for each $x \in X$, $\psi(\cdot,x)$ is measurable and for each $t \in T$, $\psi(t,\cdot)$ is continuous. Therefore, $\psi$ is jointly measurable. 
We define a correpondence $\Phi:T\to 2^X$ such that 
\[
\Phi(t)=\{x\in X:\varphi(t,x)=0\},
\] 
where $\varphi(t,x)=\psi(t,x)-x$. By the Tychonoff fixed point theorem, we can guarantee that the function $\psi(t,\cdot):X \to X$ has a fixed point. 
Consequently, for each $t\in T$, $\Phi(t)\neq \emptyset$. Since $\varphi$ is jointly measurable, one can easily check that $\Phi$ has a measurable graph. Hence by 
Aumann's measurable selection theorem, there exists a measurable function $x^{*}:T\to X$ such that $\mu$-a.e. on $T$, $x^{*}(t) \in \Phi(t)$,
i.e., $x^{*}(t)=\psi(t,x^{*}(t))\in \Psi(t,x^{*}(t)).$ This completes the proof of the theorem.
\end{proof}

\section{Applications}\label{sec:applications}

\subsection{Large abstract economies and equilibria}

Let $(T, \mathscr T, \mu)$ be a complete, finite, and positive measure space. For any correspondence 
$X:T\to 2^Y$, we define 
\[
L_X=\left\{x\in L_1(\mu, Y):x(t)\in X(t)\, \mu \mbox{-a.e.}\right\}.
\]
An \emph{abstract economy} $\Gamma$ is a quadruple $\langle (T,\mathscr T, \mu), X, P, A\rangle$, where 
\begin{enumerate}
\item[(i)] $(T, \mathscr T, \mu)$ is a measure space of agents;
\item[(ii)] $X:T\to 2^Y$ is a strategy correspondence;
\item[(iii)] $P:T\times L_X\to 2^Y$ is a preference correspondence such that 
$P(t,x)\subseteq X(t)$ for all $(t,x)\in T\times L_X$; and  
\item[(iv)] $A:T\times L_X\to 2^Y$ is a constraint correspondence such that 
$A(t,x)\subseteq X(t)$ for all $(t,x)\in T\times L_X$. 
\end{enumerate}

By definition, the preference correspondence $P$  captures the idea of interdependence.  The interpretation of these
preference correspondences is that $y\in  P(t, x)$ means that agent $t$ strictly prefers $y$ to $x(t)$ if the given strategies 
of other agents are fixed. Note that $L_X$ is the set of all joint strategies. As in Khan-Vohra \cite{KV} and 
Yannelis \cite{Yannelis:87}, we endow $L_X$ throughout the paper with the weak topology. This signifies a
natural form of myopic behaviour on the part of the agents. In particular,
an agent has to arrive at his/her decisions on the basis of knowledge of
only finitely many (average) numerical characteristics of the joint
strategies. 

\medskip
We now define the concept of an equilibrium in an abstract economy.
\begin{definition}
An \emph{equilibrium} for an abstract economy $\Gamma$ is an element $x^*\in L_X$ such that for $\mu$-a.e. on $T$, we have 
$x^*(t)\in A(t, x^*)$ and $P(t, x^*)\cap A(t,x^*)= \emptyset$. 
\end{definition}

The following theorem is a generalization of Bhowmik-Yannelis \cite{BY}, Nash \cite{Nash} and Yannelis \cite{Yannelis:87}, 
among others. It also extends 
Khan \cite{Khan} to a framework of discontinuous payoffs.

 \begin{theorem}\label{thm:main}
           Let $\Gamma$ be an abstract economy satisfying the following properties: 
\begin{enumerate}
               \item  $X: T \rightarrow 2^Y$ is an integrably bounded lower measurable correspondence and for all $t\in T$, $X(t)$
                    is a non-empty, convex, closed subset of a separable Banach space $Y$;
               \item $A: T \times L_X \to 2^Y$ is a non-empty, convex, closed-valued correspondence such that 
$A(t,\cdot): L_{X} \to 2^Y$ is upper-semicontinuous for all $t\in T$ and $A(\cdot, x):T\to 2^Y$ is lower measurable
for all $x\in L_X$; 
               \item $P: T \times L_{X}\to 2^Y$ has the property that $x(t) \notin {\rm con} P(t,x)$ for all 
$x\in L_X$ and almost all $t\in T$; 
               \item $\Psi: T \times L_{X} \to 2^Y$ satisfies the continuous inclusion property if the 
economy $\Gamma$ is purely atomic and the strong continuous inclusion property, otherwise, where $\Psi(t,x)= A(t,x)\cap 
{\rm con} P(t,x)$ for all $(t,x)\in T\times L_X$; and 
               \item either $Y$ is finite-dimensional or $\mathbb K(\Psi)(t,x)\neq \emptyset$ for all $(t,x)\in U_\Psi$.
           \end{enumerate}
    Then there exists an equilibrium for $\Gamma$.
\end{theorem}

\begin{proof}
 By  Lemma \ref{extcara:aux1} and Lemma \ref{extcara:aux2} (see Section \ref{sec:mainlemmata}), we can readily verify that
$U_\Psi\in \mathscr T\otimes \mathscr B(L_X)$. In virtue of the measurable projection theorem \cite{Yannelis:91a}, it follows that
$U_\Psi^x\in \mathscr T$, for all $x\in L_X$. Define
\[
\mathbb I_\Psi=\left\{x\in L_X: \mu(U_\Psi^{x})> 0\right\}.
\]
We consider the following two cases.

\medskip
{\bf Case 1.} $\mathbb I_\Psi=\emptyset$. In this case, we have $\mu( U_\Psi^{x})= 0$, which implies 
 $P(t, x)\cap A(t,x)= \emptyset$ for all $x\in L_X$ and almost all $t\in T$.  
Since $A$ is closed-valued and $A(\cdot, x)$ is lower measurable, we conclude that 
$A(\cdot, x)$ has a measurable graph, for all $x\in L_X$. Define $H: L_{X} \to 2^{L_{X}}$ by 
\[
H(x)= \left\{y \in L_{X}: y(t) \in A(t,x)\, \mu\mbox{-a.e.}\right\}.
\] 
In view of Lemma \textcolor{red}{4.4} of Yannelis \cite{Yannelis:87}, $H$ is non-empty valued and weakly upper-semicontinuous. Since 
$A$ is convex valued, so is $H$. Furthermore, Lemma \textcolor{red}{4.3} of Yannelis \cite{Yannelis:87} guarantees that $L_{X}$ is 
non-empty, convex, and weakly compact. Hence by Fan-Glicksberg's fixed point theorem, there exists 
$x^* \in L_{X}$  such that $x^* \in H(x^*)$, which means $x^*(t) \in A(t,x^*)$ $\mu$-a.e. Since 
$P(t, x^*)\cap A(t,x^*)= \emptyset$ $\mu$-a.e., it follows that $x^*$ is an equilibrium for the abstract economy $\Gamma$.

\medskip
{\bf Case 2.} $\mathbb I_\Psi\neq \emptyset$. By Theorem \ref{main:cara}, we conclude that there is a  Carath\'eodory-type selection 
$\psi:U_\Psi\to Y$ of $\restr{\Psi}{U_\Psi}$. Let $\Theta:T\times L_X\to 2^Y$ be a correspondence such that 
  \[
  \Theta(t, x) = \left\{
  \begin{array}{ll}
  \{\psi(t,x)\},& \mbox{if $(t, x)\in U_\Psi$;}\\[0.5em]
  A(t, x), & \mbox{otherwise.}
  \end{array}
  \right.
  \]
Clearly, $\Theta$ is non-empty and convex valued. By Lemma \ref{lem:upper}, $\Theta$ is jointly lower measurable and 
$\Theta(t,\cdot):L_X\to 2^Y$ is upper-semicontinous for all $t\in T$. Define $H: L_{X} \to 2^{L_{X}}$ by 
\[
H(x)= \left\{y \in L_{X}: y(t) \in \Theta(t,x)\, \mu\mbox{-a.e.}\right\}.
\] 
Again, by repeating the argument of Case 1, we can find an 
$x^* \in L_{X}$  such that $x^* \in H(x^*)$, which means $x^*(t) \in \Theta(t,x^*)$ $\mu$-a.e. 
We claim that $\mu(U_\Psi^{x^*})=0$. This follows from the fact that for $t\in U_\Psi^{x^*}$, we have 
$x^*(t) = \psi(t,x^*) \in \Psi(t, x^*) \subseteq {\rm con} P(t, x^*)$, which leads to a contradiction if 
$\mu(U_\Psi^{x^*})> 0$. Therefore, $\mu$-a.e. on $T$, 
$x^*(t) \in A(t,x^*)$ and $\Psi(t,x^*)= \emptyset$, which implies that 
$P(t, x^*) \cap A(t, x^*) = \emptyset$ $\mu$-a.e., i.e. $x^*$ is an equilibrium for $\Gamma$. This completes the proof.
\end{proof}

\subsection{Random games and equilibria}\label{subsec:randomgame}
Let $(\Omega , \mathscr F, \mu)$ be a complete finite measure space. We can interpret $\Omega$ as the set of states of nature of the world and assume that 
$\Omega$ is large enough to include all events that we consider to be interesting. As usual, $\mathscr F$ denotes $\sigma$-algebra of events. We denote by 
$\mathbf I$ the set of players, which may be finite or countably infinite.
\begin{definition}
    A \emph{random game} is a set $\mathscr{E}=\{(X_{i},P_{i}) : i \in \mathbf I\}$ of ordered pairs, where 
\begin{enumerate}
        \item $X_{i}$ is the strategy set of player $i\in \mathbf I$; and 
        \item  $P_{i} : \Omega \times X \to 2^{X_{i}}$ (where  $X = \Pi_{i\in \mathbf I} X_{i}$) is the random preference (choice) 
correspondence of player $i\in \mathbf I$.
    \end{enumerate}
    \end{definition}
By definition, the preference correspondence $P$ captures the idea of interdependence at every state nature of the world.
The interpretation of these preference correspondences is that $y_{i} \in P_{i}(\omega,x)$ means that player $i$ strictly prefers 
$y_{i}$ to $x_{i}$ at the state of nature $\omega$, if the (given) components of the other players are fixed.

\medskip
In what follows we define the concept of a random equilibrium in a random game $\mathscr E$.
\begin{definition}
A \emph{random equilibrium} for the game $\mathscr E$  is a measurable function $x^{*} : \Omega \to X$  such that  
$P_{i}(\omega, x^{*}(\omega)) = \emptyset$  for all $i \in \mathbf I$ and for almost all $\omega\in \Omega$.
\end{definition}
Notice that each player in the game described above is charecterized by a strategy set and a random preference correspondence. We now follow the original formulation by  Nash 
\cite{Nash} (and its generalizations by Fan \cite{Fan:52} and Browder \cite{Browder:68}, among other) where random preference correspondences are replaced by random payoff functions, i.e., real valued functions defined on $\Omega \times X$.

\medskip
 Let $\Gamma =\{(X_{i},u_{i}): i \in \mathbf I\}$ be a Nash-type random game, i.e., 
\begin{enumerate}
        \item $X_{i}$ is the strategy set of player $i\in \mathbf I$; and
        \item $u_{i} : \Omega \times X \to \mathbb R$ (where $X = \Pi_{i \in \mathbf I} X_{i})$ is the random payoff function of player $i\in \mathbf I$.
    \end{enumerate}

\medskip
 To define the concept of a random Nash equilibrium, we let $X_{-i} = \Pi_{j \neq i} X_{j}$ and denote a generic element of $X_{-i}$ by 
 $x_{-i}$. 

\begin{definition}
A \emph{random Nash equilibrium} for $\Gamma$ is a measurable function $x^{*}: \Omega \to X $ such that for all $i\in \mathbf I$ and for 
almost all $\omega\in \Omega $, we have 
\[
         u_{i}(\omega,x^{*}(\omega)) = \max\left\{u_{i}(\omega, y_{i}, x_{-i}^*(\omega)):{y_i} \in X_{i}\right\}.
 \]
\end{definition}
  Below, we proceed to establish a random equilibrium existence result.
       \begin{theorem}\label{thm:extrangame}
           Let $\mathscr{E} = \{(X_{i},P_{i}): i \in \mathbf I \}$ be a random game satisfying for each $ i\in \mathbf I$ the following properties: \begin{enumerate}
               \item each $X_{i}$ is a nonempty, compact, and convex subset of a separable Banach space $Y$;
               \item for every measurable function $x : \Omega \to X $ we have $x_{i} (\omega) \notin {\rm con} P_{i} (\omega, x(\omega))$ $\mu$-a.e.
               on $\Omega $;
               \item the correspondence $P_i$ satisfies the continuous inclusion property if the measure space is atomic and the strong continuous 
inclusion property, otherwise; and 
               \item either $Y$ is finite-dimensional or $\mathbb K(P_i)(\omega,x)\neq \emptyset$ for all $(\omega, x)\in U_{P_i}$.
           \end{enumerate}
    Then there exists a random equilibrium for $\mathscr E$.
       \end{theorem}

\begin{proof}
For each $i \in \mathbf I$, define a correspondence $\Psi _{i} : \Omega \times X \rightarrow 2^{X_{i}}$ by letting $\Psi _{i} (\omega, x) = {\rm con}P_{i} (\omega,x)$. 
Pick an $i\in \mathbf I$. 
From Theorem \ref{main:cara}, it follows that there exists a Carath\'eodory-type selection $\psi_i:U_{\Psi_i}\to Y$ of $\restr{\Psi_i}{U_{\Psi_i}}$
Let $F_i:\Omega\times X\rightarrow 2^{X_i}$ be a correspondence such that 
  \[
  F_i(\omega, x) = \left\{
  \begin{array}{ll}
  \{\psi_i(\omega,x)\},& \mbox{if $(\omega, x)\in U_{\Psi_i}$;}\\[0.5em]
  X_i, & \mbox{otherwise.}
  \end{array}
  \right.
  \]
Clearly, $F_i$ is nonempty, closed and convex valued.  By Lemma \ref{lem:upper}, $F_i$ is jointly lower measurable and 
$F_i(\omega,\cdot):X\to 2^{X_i}$ is upper-semicontinous for all $\omega\in \Omega$.
Next, we consider a correspondence $\Lambda: \Omega \times X \rightarrow 2^X$ defined by $\Lambda(\omega,x) = \Pi_{i \in \mathbf I} F_{i}(\omega,x)$. 
Clearly, $\Lambda$ is nonempty, closed and convex valued. Since each $F_{i}$ is jointly lower measurable, it follows from Theorem \textcolor{red}{17.28} of Aliprantis-Border \cite{AB} that 
$\Lambda$ is jointly lower measurable as well and the correspondence $\Lambda(\omega,\cdot) : X \rightarrow 2^X$ is 
upper-semicontinuous for each $\omega\in \Omega$. By Corollary \textcolor{red}{2.11} of Yannelis-Rustichini \cite{YR}, there is
a measurable function $x^{*} : \Omega \to X$ such that 
$x^{*}(\omega) \in \Lambda(\omega,x^{*}(\omega))$ $\mu$-a.e. on $\Omega$. If $(\omega, x^{*}(\omega)) \in U_{\Psi_i}$ for all $\omega\in 
\Omega_0\in \mathscr F$ with $\mu(\Omega_0) >0$, then by definition of $F_{i}$, we have $x_i^{*}(\omega) = \psi_{i}(\omega, x^{*}(\omega)) \in {\rm con}\, 
P_{i}(\omega,x^{*}(\omega))$, contrary to Assumption 2. Thus, $(\omega,x^{*}(\omega)) \notin U_{\Psi_i}$ holds for $\mu$-a.e. on $\Omega$ 
and all $i \in \mathbf I$. In other words, we have $\Psi_{i}(\omega,x^{*}(\omega)) =  \emptyset$ $\mu$-a.e. on $\Omega$ and all $i\in \mathbf I$. This means that 
$x^{*}:\Omega \to X$ is a random equilibrium for $\mathscr E$, and this completes the proof of the theorem. 
\end{proof}

As a corollary to Theorem \ref{thm:extrangame}, we obtain a generalised random version of Nash's equilibrium existence result. To this end, 
for each $i\in \mathbf I$, $\omega\in \Omega$ and $x\in X$, we define
\[
 P_{i}^{u_i}(\omega,x) = \{ y_{i} \in X_{i} : v_{i}(\omega,x,y_{i}) >0\}, 
 \]
 where $v_{i}(\omega,x,y_{i})$ = $u_{i}(\omega,y_{i},x_{-i})- u_{i}(\omega, x)$. Before we state our corollary, we first note that 
if $u_i(\omega,\cdot)$ is continuous for each $\omega\in \Omega$ and $u_i(\cdot, x)$ is measurable for each $x\in X$ then $P_i^{u_i}$ satisfies the strong 
continuous inclusion property. Indeed, $u_i(\omega,\cdot)$ is continuous implies that $P_i^{u_i}(\omega,\cdot)$ has an open graph in $X\times X_i$
and hence, it is lower-semicontinuous. We now prove that $P_i^{u_i}$ is jointly lower measurable. To do this, since for each fixed $\omega\in \Omega$ 
the function $u_i(\omega,\cdot)$ is continuous and for each fixed $x\in X$ the function $u_i(\cdot,x)$ is measurable, it follows from a standard 
result that $u_i$ is jointly measurable. Choose an open subset $V$ of a separable Banach space $Y$. We must show that 
\[
\{(\omega, x)\in \Omega\times X: P_i^{u_i}(\omega,x) \cap V \neq \emptyset\}\in \mathscr F\otimes \mathscr B(X).
\] 
Let $D$ be a countable dense subset of $V$. Observe that 
\begin{align*}
                \{(\omega, x)\in \Omega\times X :P_i^{u_i}(\omega,x) \cap V \neq \emptyset\} &= \{(\omega, x)\in \Omega\times X :v_i(\omega,x, y_i) > 0
\mbox{ for some } y_i\in V\}\\
                &=\{(\omega, x)\in \Omega\times X :v_i(\omega,x, y_i) > 0
\mbox{ for some } y_i\in D\}\\
                & =\bigcup_{y_i \in D}\{(\omega, x)\in \Omega\times X :v_i(\omega,x, y_i) > 0\}
            \end{align*}
         and the latter set belongs to $\mathscr F\otimes \mathscr B(X)$ since for each fixed $y_i\in D$ the function $v_i(\cdot,\cdot,y_i)$ 
is $\mathscr F\otimes \mathscr B(X)$-measurable. Hence, $P_i^{u_i}$ is jointly lower measurable. Therefore, $P_i^{u_i}$ satisfies the strong continuous inclusion 
property whenever $u_i(\omega,\cdot)$ is continuous for each $\omega\in \Omega$ and $u_i(\cdot, x)$ is measurable for each $x\in X$. 
\begin{corollary}\label{cor:3.3}
    Let $\Gamma = \{(X_{i},u_i) : i \in \mathbf I \}$ be a Nash-type random game satisfying for each $i\in \mathbf I$ the following assumptions: \begin{enumerate}
        \item $X_{i}$ is a nonempty, compact and convex subset of a separable Banach space $Y$;
        \item for each $(\omega,x_{-i}) \in \Omega\times X_{-i}$ the function $u_{i}(\omega,x_{i},x_{-i})$ is 
        quasi-concave in $x_{i}$;  
        \item the correspondence $P_i^{u_i}$ satisfies the continuous inclusion property if the measure space is atomic and the strong continuous 
inclusion property, otherwise; and 
 \item either $Y$ is finite-dimensional or $\mathbb K(P_i^{u_i})(\omega,x)\neq \emptyset$ for all $(\omega, x)\in U_{P_i^{u_i}}$.
    \end{enumerate}
    Then there exists a random Nash equilibrium for $\Gamma$.
\end{corollary}

\begin{proof}
         It follows from Assumption 2 that each $P_{i}^{u_i}$ is convex valued, and clearly for any measurable function $ x: \Omega \to X$ we have 
$x_{i}(\omega) \notin P_{i}^{u_i}(\omega,x(\omega))={\rm con}\,P_{i}^{u_i}(\omega,x(\omega))$ for all $\omega\in \Omega$. Hence, the random game 
$\mathscr E= \{(X_{i},P_{i}^{u_i}):i \in \mathbf I\}$ satisfies the assumption of Theorem \ref{thm:extrangame} and therefore $\mathscr E$ has a 
random equilibrium. That is, there exists a measurable function $x^{*}: \Omega \to X$ such that $P_{i}^{u_i}(\omega,x^{*}(\omega))= \emptyset$ 
$\mu$-a.e. on $\Omega$ and all $i \in \mathbf I $. But this implies that 
\[
         u_{i}(\omega,x^{*}(\omega)) = \max\left\{u_{i}(\omega, y_{i}, x_{-i}^*(\omega)):{y_i} \in X_{i}\right\}
 \]
  $\mu$-a.e. on $\Omega$ and all $i \in \mathbf I $. Hence, $x^{*}$ is a random Nash equilibrium for the game $\Gamma =\{(X_{i},u_{i}):i \in \mathbf I\}$.
    \end{proof}

\subsection{Bayesian games}
We now turn to the problem of the existence of equilibrium points for Bayesian games. To do this, let $(\Omega,\mathscr F,\mu)$ be a complete 
finite probability space, denoting the space of states of nature of the world. We denote by $\mathbf I$ the set of players, which may be finite or 
countably infinite.

\begin{definition}
    A \emph{Bayesian game} is a set $\mathscr{G}=\{(X_{i},u_{i},\mathscr F_{i},q_{i}) : i \in \mathbf I\}$ of quadruples such that 
    \begin{enumerate}
        \item $X_{i}$ is the strategy set of player $i\in \mathbf I$; 
        \item $u_{i}: \Omega \times X \to \mathbb R$ (where $X = \Pi_{i \in \mathbf I} X_{i}$) is the random payoff function of player $i\in \mathbf I$;
        \item $\mathscr F_{i}$ is a measurable partition of $\Omega$ denoting the (private) information available to player $i\in \mathbf I$; and
        \item $q_{i}: \Omega \to (0,1]$ is the prior probability density of player $i\in \mathbf I$, i.e., $q_{i}$ is a measurable function having the property 
that $\int_{\Omega} q_{i}(\omega) d \mu(\omega) = 1$.
    \end{enumerate}
    
\end{definition}
%As in Aumann \cite{Aumann} or Myerson \cite{Mayerson}, it is assumed that the game $\mathscr{G}=\{(X_{i},u_{i},\mathscr F_{i},q_{i}) : i \in \mathbf I\}$ is 
%common knowledge, i.e., every player knows $\mathscr{G}$, every player knows that every player knows $\mathscr{G}$, every player knows 
%that every player knows that every player knows $\mathscr{G}$ and so on.

\medskip
We consider the case where the information set of each player $i\in \mathbf I$ is the same, i.e., $\mathscr F_{i}=\mathscr F$ for each $i \in \mathbf I$. 
Denote by $\mathbf E(\omega)$ the event in $\mathscr F$ which contains the realized state of nature $\omega \in \Omega$, and suppose that 
$q_{i}(\mathbf E(\omega)) >0$ for all $i \in \mathbf I$ and all $\omega\in \Omega$. Given $\mathbf E(\omega)$, the conditional expected utility of player 
$i\in \mathbf I$ is the function $h_{i}: \Omega \times X \to \mathbb R$ defined by
\[
    h_{i}(\omega,x) = \int_{\mathbf E(\omega)} q_{i}(t|\mathbf E(\omega ))u_{i}(t,x)d\mu(t),
\]
where
\[
  q_{i}(t|\mathbf E(\omega)) = \left\{
  \begin{array}{ll}
  0,& \mbox{if $t\notin \mathbf E(\omega)$;}\\[0.5em]
    \frac{q_{i}(t)}{\int_{\mathbf E(\omega)}q_{i}(s)d\mu(s)}, & \mbox{if $t\in \mathbf E(\omega)$.}
  \end{array}
  \right.
  \]
\begin{definition}
A \emph{Bayesian equilibrium} for a Bayesian game $\mathscr G$
is a function $x^{*}: \Omega \to X$ such that each $x^{*}_{i}$ is $\mathscr F$-measurable and for each $i \in \mathbf I$ we have
\begin{equation*}
    h_{i}(\omega,x^{*}(\omega)) =\,\max\left\{h_{i}(\omega, y_{i}, x_{-i}^*(\omega)):{y_i} \in X_{i}\right\}
\end{equation*}
for almost all $\omega \in \Omega$.
\end{definition}
Note that if $u_{i}(\omega,\cdot)$ is continuous and $u_{i}$ is integrably bounded then by virtue of the Lebesgue dominated convergence theorem, we can  
conclude that the function
    \begin{equation*}
        h_{i}(\omega,\cdot) = \int_{\mathbf E(\omega)} q_{i}(t|\mathbf E(\omega))u_i(t,\cdot)d\mu(t)
    \end{equation*}
    is continuous, where
\[
  q_{i}(t|\mathbf E(\omega)) = \left\{
  \begin{array}{ll}
  0,& \mbox{if $t\notin E(\omega)$;}\\[0.5em]
    \frac{q_{i}(t)}{\int_{\mathbf E(\omega)}q_{i}(s)d\mu(s)}, & \mbox{if $t\notin E(\omega)$.}
  \end{array}
  \right.
  \]
Furthermore, it can be easily seen that each function $h_{i}(\cdot,x)$ is $\mathscr F$-measurable.
For each $i\in \mathbf I$, $\omega\in \Omega$ and $x\in X$, we define
\[
 Q_{i}^{u_i}(\omega,x) = \{ y_{i} \in X_{i} : v_{i}(\omega,x,y_{i}) >0\}, 
 \]
 where $v_{i}(\omega,x,y_{i}) = h_{i}(\omega,y_{i},x_{-i})- h_{i}(\omega, x)$. Note that $Q_{i}^{u_i}$ is the same as $P_{i}^{h_i}$, as 
defined in the previous section. Consequently, if $u_{i}(\omega,\cdot)$ is continuous and $u_{i}$ is integrably bounded then, as $h_i(\omega,\cdot)$ is continuous for all $\omega\in \Omega$ and $h_i(\cdot,
x)$ is ${\mathscr F}$-measurable for all $\omega\in \Omega$, it follows from our previous section that 
$Q_i^{u_i}=P_{i}^{h_i}$ has the strong continuous inclusion property. We now ready to establish our Bayesian equilibrium existence theorem without assuming 
the continuity of $u_i(\omega,\cdot)$ for all $i\in \mathbf I$. 
\begin{theorem}\label{thm:3.1}
    Let $\mathscr{G}=\{(X_{i},u_{i},\mathscr F_{i},q_{i}) : i \in \mathbf I\}$ be a Bayesian game satisfying for each $i\in \mathbf I$ the following properties:
\begin{enumerate}
        \item each $X_{i}$ is a nonempty, compact and convex subset of a separable Banach space $Y$;
        \item for each $\omega \in \Omega$ and each $x_{-i} \in X_{-i}(=\Pi_{j \neq i} X_{j})$ the function $u_{i}(\omega,x_{i},x_{-i})$ is 
        concave in $x_{i}$; 
        \item the correspondence $Q_i^{u_i}$ satisfies the continuous inclusion property if the measure space is atomic and the strong continuous 
inclusion property, otherwise; and 
 \item either $Y$ is finite-dimensional or $\mathbb K(Q_i^{u_i})(\omega,x)\neq \emptyset$ for all $(\omega, x)\in U_{Q_i^{u_i}}$.
    \end{enumerate}
Then the game $\mathscr{G}$ has a Bayesian equilibrium.
\end{theorem}

\begin{proof}
 For each $\omega\in \Omega$ and each $x_{-i} \in X_{-i}$, by virtue of Condition \textcolor{red}{2}, it can be show that 
the function $h_{i}(\omega,\cdot,x_{-i})$ is concave. We consider the Bayesian game $\mathscr{G}=\{(X_{i},u_{i},\mathscr F_i,q_{i}): i \in 
\mathbf I\}$ as a Nash-type random game $\Gamma=\{(X_{i},h_{i}): i \in \mathbf I\}$. Since for each $i\in \mathbf I$ and $x\in X$ the function $h_i(\cdot,x)$ 
is $\mathscr F$-measurable, we conclude that the existence of a random Nash equilibrium for the game $\Gamma$ implies that existence 
of a Bayesian equilibrium for the game $\mathscr G$.
It can be easily seen that the random game $\Gamma$ satisfies all the 
conditions of Corollary \ref{cor:3.3}, and consequently, the game $\Gamma$ has a random Nash equilibrium. Hence, there exists an 
$\mathscr F$-measurable function $x^{*}: \Omega \to X$ such that \begin{equation*}
    h_{i}(\omega,x^{*}(\omega)) =\,\max\left\{h_{i}(\omega, y_{i}, x_{-i}^*(\omega)):{y_i} \in X_{i}\right\}
\end{equation*}
$\mu$-a.e. on $\Omega$ and all $i \in \mathbf I$. In other words, $x^{*}$ is a Bayesian equilibrium for the game $\mathscr{G}=
\{(X_{i},u_{i},\mathscr F_i,q_{i}): i \in \mathbf I\}$, and the proof of the theorem is finished.
\end{proof}

%%%%%%%%%%%%%%%%%%%%%%%%%%%%%%%%%%%%%%%%%%%%%
\subsection{Asymmetric Bayesian games}
%%%%%%%%%%%%%%%%%%%%%%%%%%%%%%%%%%%%%%%%%%%%

We now turn to the problem of the existence of equilibrium points for asymmetric Bayesian games. To do this, let $(\Omega,\mathscr F,\mu)$ be a complete 
finite, separable measure space, denoting the space of states of nature of the world. We denote by $\mathbf I$ the set of players, where 
$\mathbf I$ can be finite or countably infinite, and by $Y$ a separable Banach space.  

\begin{definition}
An \emph{asymmetric Bayesian game} (or a \emph{game with asymmetric information}) is a set
\[
\mathscr G^a = \{(X_i, u_i, \mathscr{F}_i, q_i) : i \in \mathbf I\},
\]
where
\begin{enumerate}
    \item $X_i: \Omega \rightarrow 2^Y$ is the \textit{action set-valued function} of player $i\in \mathbf I$, where $X_i(\omega)$ is the set of actions available to player $i\in \mathbf I$ when the state is $\omega\in \Omega$;
    \item for each $\omega \in \Omega$, $u_i(\omega, \cdot) : \prod_{j \in \mathbf I} X_j(\omega) \to \mathbb{R}$ is the \textit{state-dependent utility function} of player $i\in \mathbf I$;

    \item $\mathscr{F}_i$ is a sub $\sigma$-algebra of $\mathscr{F}$ which denotes the \textit{private information} of player $i\in \mathbf I$; and 
    \item $q_i : \Omega \to \mathbb{R}_{++}$ is the \textit{prior} of player $i\in \mathbf I$, which is a Radon-Nikodym derivative such that $\int q_i(\omega) \, d\mu(\omega) = 1$.
\end{enumerate}
\end{definition}

Let $L_{X_i}$ denote the set of all Bochner integrable and $\mathscr{F}_i$-measurable selections from the action correspondence $X_i: \Omega \rightarrow 2^Y$ of player $i\in \mathbf I$, i.e.,
\[
L_{X_i} = \{\widetilde{x}_i \in L_1 (\mu, Y) : \widetilde{x}_i \text{ is } \mathscr{F}_i\text{-measurable and } \widetilde{x}_i(\omega) \in X_i(\omega) \text{ $\mu$-a.e.}\}.
\]
The typical element of $L_{X_i}$ is denoted as $\widetilde{x}_i$, while that of $X_i(\omega)$ as $x_i(\omega)$ (or $x_i$).  A \emph{strategy} for 
player $i\in \mathbf I$ is an 
element $\widetilde{x}_i$ in $L_{X_i}$. Let $L_X = \prod_{i\in \mathbf I} L_{X_i} \mbox{ and } L_{X_{-i}} = \prod_{j\neq i} L_{X_j}$.

\medskip
Throughout the rest of this subsection, we assume that for each $i\in \mathbf I$ there exists a finite or countable partition $\pi_i$ of $\Omega$. Moreover, the $\sigma$-algebra 
$\mathscr{F}_i$ is assumed to be generated by $\pi_i$. For each $\omega \in \Omega$, let $\mathbf E_i(\omega) \in \pi_i$ denote the smallest set in $\mathscr{F}_i$ containing 
$\omega$, and we assume that 
\[
\int_{\mathbf{E}_i(\omega)} q_i(t) \, d\mu(t) > 0
\]
for all $i\in \mathbf I$. For each $\omega \in \Omega$, the \textit{conditional $($interim$)$ expected utility function} of 
player $i$, $h_i(\omega, \cdot, \cdot) : L_{X_{-i}} \times X_i(\omega) \to \mathbb{R}$, is defined as:
\[
h_i(\omega, \widetilde{x}_{-i }, x_i) = \int_{{\mathbf E}_i(\omega)} u_i(t, \widetilde{x}_{-i}(t), x_i) q_i(t|{\mathbf E}_i(\omega)) \, d\mu(t),
\]
where
\[
q_i(t | {\mathbf E}_i(\omega)) =
\begin{cases}
    0, & \text{if } t\notin {\mathbf E}_i(\omega);\\
    \frac{q_i(t)}{\int_{\mathbf E_i(\omega)} q_i(s) \, d\mu(s)}, & \text{if } t\in {\mathbf E}_i(\omega).
\end{cases}
\]
The function $\nu_i(\omega, \widetilde{x}_{-i}, x_i)$ is interpreted as the conditional expected utility of player $i$ using the action $x_i$ when the state is $\omega$ 
and the other agents employ the strategy profile $\widetilde{x}_{-i}$, where $\widetilde{x}_{-i}$ is an element of $L_{X_{-i}}$.

\begin{definition}
A \emph{Bayesian Nash equilibrium} for $\mathscr G^a$ is a strategy profile $\widetilde{x}^* \in L_X$ such that for all $i \in \mathbf I$ and $\mu\text{-a.e.}$ on $\omega$,
\[
h_i(\omega, \widetilde{x}^*_{-i}, \widetilde{x}^*_i(\omega)) = \max\{h_i(\omega, \widetilde{x}^*_{-i}, y_i):y_i \in X_i(\omega)\}. 
\]
\end{definition}

\medskip
Before we study the existence of a Bayesain Nash equilibrium for $\mathscr G^a$, we first define a correspondence $Q_i^{u_i} : \Omega \times 
L_{X_{-i}} \rightarrow 2^Y$ for each $i\in \mathbf I$ by letting 
\[   
 Q_i^{u_i}(\omega, \widetilde{x}_{-i}) = \{y_i \in X_i(\omega) : \nu_i(\omega, \widetilde{x}, y_i) > 0\},
\]
where $\nu_i(\omega, \widetilde{x}, y_i):=h_i(\omega, \widetilde{x}_{-i}, y_i) - h_i(\omega, \widetilde{x}_{-i}, \widetilde{x}_i(\omega) )$. 
In what follows, we show under some continuity and measurability assumptions, $Q_i^{u_i}$ satisfies the strong continuous inclusion 
property. To this end, we assume the following:

    \begin{enumerate}
        \item  For each $\omega \in \Omega$, $u_i(\omega, \cdot, \cdot) : \prod_{j\neq i} X_j(\omega) \times X_i(\omega) \to \mathbb{R}$ is continuous, where 
$X_j(\omega)$ $(j \neq i)$ is endowed with the weak topology and $X_i(\omega)$ with the norm topology.
        \item For each $x \in \prod_{j\in \mathbf I} Y_j$ with $Y_j = Y$, $u_{i}(\cdot, x): \Omega \to \mathbb{R}$ is $\mathscr{F}_i$-measurable.
        \item $u_i$ is integrably bounded.
        \end{enumerate}

Under the above assumptions, it follows from a standard argument (see, for example, \cite{Balder}, \cite[Lemma A1]{KY}) that for each $\omega \in \Omega$, $h_i(\omega, \cdot, \cdot) : L_{X_{-i}} 
\times X_i(\omega) \to \mathbb{R}$ is continuous, where $L_{X_{j}}$ $(j \neq i)$ is endowed with the weak topology and $X_i(\omega)$ 
with the norm topology. Furthermore, for each $(\widetilde{x}_{-i}, x_i) \in L_{X_{-i}} \times X_i(\omega)$, $h_i(\cdot, \widetilde{x}_{-i}, x_i) : 
\Omega \to \mathbb{R}$ is $\mathscr{F}_i$-measurable. 
 It follows from the continuity of $h_i(\omega,\cdot,\cdot)$ that for each $\omega \in \Omega$, 
$Q_i^{u_i}(\omega, \cdot)$ has an open graph in $L_{X_{-i}} \times X_i(\omega)$, where $L_{X_{j}}$ $(j \neq i)$ is endowed with the weak topology and $X_i(\omega)$ 
with the norm topology. Therefore, we can conclude that $Q_i^{u_i}(\omega,\cdot)$ is lower-semicontinuous. The $\mathscr F_i\otimes 
\mathscr B(L_{X_{-i}})$-measurability of $Q_i^{u_i}$ can be done 
analogous to Subsection \ref{subsec:randomgame}. Consequently, under \textcolor{red}{1-3},  
$Q_i^{u_i}$ has the strong continuous inclusion property. We now ready to state our Bayesian equilibrium existence theorem without assuming 
the fact that $u_i(\omega,\cdot)$ is continuous for all $\omega\in \Omega$ and $i\in \mathbf I$.

\begin{theorem}\label{thm:asybaye}
    Let $\mathscr{G}=\{(X_{i},u_{i},\mathscr F_{i},q_{i}) : i \in \mathbf I\}$ be a Bayesian game satisfying for each $i\in \mathbf I$ the following properties:
\begin{enumerate}
        \item  $X_i: \Omega \rightarrow 2^Y$ is a non-empty, convex, compact-valued and integrably bounded correspondence having a 
$\mathscr{F}_i$-measurable graph, i.e., $G_{X_i} \in \mathscr{F}_i \otimes \mathscr B(Y)$;
        \item for each $\omega \in \Omega$  and ${x_{-i} \in \prod_{j \neq i} X_{j}(\omega)}$, ${u_i(\omega, x_{-i}, \cdot) : X_i(\omega) \to \mathbb{R}}$ is concave; 
        \item the correspondence $Q_i^{u_i}$ satisfies the continuous inclusion property if the measure space is atomic and the strong continuous 
inclusion property, otherwise; and 
 \item either $Y$ is finite-dimensional or $\mathbb K(Q_i^{u_i})(\omega,\widetilde{x}_{-i})\neq \emptyset$ for all $(\omega, \widetilde{x}_{-i})\in U_{Q_i^{u_i}}$.
    \end{enumerate}
\end{theorem}

\begin{proof}
Pick an $i\in \mathbf I$ and let $\Phi_i= Q_i^{u_i}$. From Theorem \ref{main:cara}, it follows that there exists a Carath\'eodory-type selection 
$\psi:U_{\Phi_i}\to Y$ of $\restr{\Phi_i}{U_{\Phi_i}}$. 
Let $F_i:\Omega\times  L_{X_{-i}}\rightarrow 2^{X_i}$ be a correspondence such that 
  \[
  F_i(\omega, \widetilde{x}_{-i}) = \left\{
  \begin{array}{ll}
  \{\psi_i(\omega,\widetilde{x}_{-i})\},& \mbox{if $(\omega, \widetilde{x}_{-i})\in U_{\Phi_i}$;}\\[0.5em]
  X_i(\omega), & \mbox{otherwise.}
  \end{array}
  \right.
  \]
Clearly, $F_i$ is nonempty, closed and convex valued. In virtue of Lemma \ref{lem:upper}, $F_i(\cdot, \widetilde{x}_{-i})$ is lower measurable and 
$F_i(\omega,\cdot):X\to 2^{X}$ is upper-semicontinous in the sense that $\left\{\widetilde{x}_{-i}\in L_{X_{-i}}: F_i(\omega,\widetilde{x}_{-i}) \subset V\right\}$  is a weakly 
open subset of $L_{X_{-i}}$ for every open subset $V$ of $Y$.
For each $i\in \mathbf I$, define $H_i: L_{X_{-i}} \rightarrow 2^{L_{X_i}}$ by
\[
H_i(\widetilde{x}_{-i}) = \{ \widetilde{y}_{i} \in L_{X_{i}} : \widetilde{y}_{i}(\omega) \in F_i(\omega, \widetilde{x}_{-i})\, \text{$\mu$-a.e.} \}.
\]
Since for each \( \widetilde{x}_{-i} \in L_{X_{-i}}, F_i(\cdot, \widetilde{x}_{-i}) \) has a measurable graph by virtue of the Aumann measurable selection theorem, there exists an \( \mathscr{F}_i \)-measurable function \( g_i : \Omega \to Y \) such that \( g_i(\omega) \in F_i(\omega, \widetilde{x}_{-i}) \) \( \mu \)-a.e. Since for each \( (\omega, \widetilde{x}_{-i}) \in 
\Omega \times L_{{X}_{-i}}$, $F_i(\omega, \widetilde{x}_{-i}) \subseteq X_i(\omega) \) and \( X_i(\cdot) \) is integrably bounded, it follows that \( g_i \in L_{X_i}$. Hence, $ g_i \in 
H_i(\widetilde{x}_{-i})$, i.e., $H_i$ is non-empty-valued. By Diestel's Theorem (see Khan \cite{Khan:84}),
$L_{X_i}$ is a weakly compact subset of $L_1(\mu, Y)$ (see Theorem 3, \cite{Yannelis:91b}).

Since the weak topology for a weakly compact subset of a separable Banach space is metrizable [\cite{DS}, p. 434], we can conclude that $L_{X_i}$ is metrizable, and since 
$\mathbf I$ is countable, so is $ L_X$. It follows from the Fatou Lemma in infinite-dimensional spaces (see, for example, \cite{Yannelis:91b}) that for each $i\in \mathbf I$, $H_i$ 
is (weakly) upper-semicontinuous, and it is obviously convex and closed-valued. Define $\Psi : L_{X} \rightarrow 2^{L_X}$ by letting 
\[
\Psi(x) = \prod_{i\in \mathbf I} H_i(\widetilde{x}_{-i}).
\]
Clearly, $L_X$ is non-empty, compact, convex, and $\Psi:L_X\rightrightarrows L_X$ is non-empty, convex, closed-valued and (weakly) upper-semicontinuous. 
By the Fan-Glicksberg fixed point theorem, there exists $\widetilde{x}^* \in L_X$ such that $\widetilde{x}^\ast \in \Psi(x^\ast)$. If $(\omega, \widetilde{x}^{*}_{-i}) \in 
U_{Q_i^{u_i}}$ for all $\omega\in \Omega_0\in \mathscr F_i$ with $\mu(\Omega_0) >0$, then by definition of $H_{i}$, we have $x_i^{*}(\omega) = 
\psi_{i}(\omega, \widetilde{x}^{*}_{-i}) \in Q_{i}^{u_i}(\omega,\widetilde{x}^{*}_{-i})$, a contradiction. Thus, $(\omega,\widetilde{x}^{*}_{-i}) \notin U_{Q_i^{u_i}}$ holds for $\mu$-a.e. on $\Omega$ 
and all $i \in \mathbf I$. In other words, we have $Q_{i}^{u_i}(\omega,\widetilde{x}^{*}_{-i}) =  \emptyset$ $\mu$-a.e. on $\Omega$ and all $i\in \mathbf I$. This means that 
$\widetilde{x}^*$ is a Bayesian Nash equilibrium for $\mathscr G^a$, and this completes the proof of the theorem. 
\end{proof}

\subsection{Random maximal elements}
Let $X$ be a nonempty subset of a linear topological space. Let $P : X \to 2^{X}$ be a preference correspondence. We read $y \in P(x)$ as ``y is strictly preferred to x". For instance if $>$ is a binary relation on $X$ one may define $P: X \to 2^{X}$ by $P(x) =\{y \in X : y > x\}$. The correspondence $P : X \to 2^{X}$ is said to have a maximal element if there exists $\Bar{x} \in X$ such that $P(\Bar{x}) = \emptyset$. Several results on the existence of maximal elements with applications to equilibrium theory have been given in the literature (see for instance Fan \cite{Fan:52}, He-Yannelis \cite{HY:2017}, Sonnenschein \cite{Sonn71} and Yannelis-Prabhakar \cite{YD:83}, among others). Notice that the above preference correspondences 
 need not be representable by utility functions. We will now allow our preference correspondence to depend on the states of nature, i.e., we allow for random preferences.

\medskip
Let $(\Omega,\mathscr F,\mu)$ be a complete finite measure space. We interpret $\Omega$ as the states of nature of the world and suppose that $\Omega$ is large enough to include all the events that we consider to be interesting. A random preference correspondence $P$ is a mapping from $\Omega \times X$ into $2^X$. We read $y \in P(\omega,x)$ as ``$y$ is strictly preferred to $x$ at the state of nature $\omega$". We can now introduce the concept of a random maximal element which is the natural analogue of the ordinary (deterministic) notion of a maximal element. The correspondence $P : \Omega \times X \to 2^X$ 
is said to have a \emph{random maximal element} if there exists a measurable function $x^*:\Omega \to X$ such that $P(\omega, x^*(\omega)) = \emptyset$ $\mu$-a.e. on $\Omega$.

\medskip
 The following theorem on the existence of random maximal elements below generalize the ordinary (deterministic) maximal elements results given in \cite{Fan:52, HY:2017, Sonn71, YD:83}. 
 \begin{theorem}\label{thm:1}
 Let $(\Omega,\mathscr F,\mu)$ be a complete finite measure space and $X$ be a nonempty, compact and convex subset of a separable Banach space $Y$. Let 
$P:\Omega \times X \rightrightarrows X$ be a correspondence $($possibly empty-valued$)$ such that: 
 \begin{enumerate}
     \item for every open subset $V$ of $X$, $\{(\omega,x) \in \Omega\times X : {\rm con} P(\omega,x) \cap V \neq \emptyset\}$ belongs to $\mathscr F \otimes \mathscr B(X)$;
     \item for every measurable function $x :\Omega\to X$ we have $x(\omega) \notin {\rm con}\, P(\omega,x(\omega))$ $\mu$-a.e. on $\Omega$; 
     \item the correspondence $P$ satisfies the continuous inclusion property if the measure space is atomic and the strong continuous 
inclusion property, otherwise; and 
 \item either $Y$ is finite-dimensional or $\mathbb K(P)(\omega,x)\neq \emptyset$ for all $(\omega, x)\in U_P$.
 \end{enumerate}
 Then there exists a measurable function $x^*:\Omega\to X$ such that $P(\omega,x^*(\omega))=\emptyset$ $\mu$-a.e. on $\Omega$.
 \end{theorem}
\begin{proof}
Let $\Psi: \Omega \times X \to 2^X$ be such that $\Psi (\omega, x) = {\rm con}P(\omega,x)$ for all $(\omega, x)\in \Omega\times X$. 
From Theorem \ref{main:cara}, it follows that there exists a Carath\'eodory-type selection $\psi:U_{\Psi}\to Y$ of $\restr{\Psi}{U_{\Psi}}$.
Let $G:\Omega\times X\rightarrow 2^{X}$ be a correspondence such that 
  \[
  G(\omega, x) = \left\{
  \begin{array}{ll}
  \{\psi(\omega,x)\},& \mbox{if $(\omega, x)\in U_{\Psi}$;}\\[0.5em]
  X, & \mbox{otherwise.}
  \end{array}
  \right.
  \]
Clearly, $G$ is nonempty, closed and convex valued.  By Lemma \ref{lem:upper}, $G$ is jointly lower measurable and 
$G(\omega,\cdot):X\to 2^{X}$ is lower-semicontinous for all $\omega\in \Omega$.
Therefore, $G$ satisfies the strong continuous inclusion property and consequently, by 
Theorem \ref{thm:fixed}, $G$ has a random fixed point, i.e., there exists a measurable function $x^{*} : \Omega \to X$ such that 
$x^{*}(\omega) \in G(\omega,x^{*}(\omega))$ $\mu$-a.e. on $\Omega$. Analogous to Theorem \ref{thm:extrangame}, one can show 
that $P(\omega,x^*(\omega))=\emptyset$ $\mu$-a.e. on $\Omega$. 
\end{proof}

\end{document}